\documentclass[journal,11pt,onecolumn,oneside, letter]{article}

\def\BibTeX{{\rm B\kern-.05em{\sc i\kern-.025em b}\kern-.08em
		T\kern-.1667em\lower.7ex\hbox{E}\kern-.125emX}}

\usepackage{float}

\usepackage{ifpdf}
\usepackage{cite}
\usepackage{amsmath}
\usepackage{amssymb}
\usepackage{amsthm}
\usepackage{url}
\usepackage{subfig}
\usepackage{fixltx2e}
\usepackage{graphics} 
\usepackage{graphicx}
\usepackage{epsfig} 
\usepackage{stfloats}
\usepackage{epstopdf}
\usepackage{hyperref}
\hypersetup{
	colorlinks=true,
	linkcolor=blue,
	filecolor=blue,
	urlcolor=blue,
	citecolor = blue,
}
\usepackage[textwidth=0.7\paperwidth,bottom=35mm,]{geometry}

\usepackage{tikz}
\usetikzlibrary{arrows} 
\usetikzlibrary{calc} 

\tikzset{actor/.style={
		rectangle,
		minimum size=6mm,
		very thick,
		draw=red!50!black!50,
		top color=white,
		bottom color=red!50!black!20
	},
	arrow/.style={
		-latex, thick, shorten <=2pt,shorten >=2pt
	}
}

\allowdisplaybreaks

\newtheorem{theorem}{Theorem}
\newtheorem{lemma}{Lemma}

\newtheorem{proposition}{Proposition}
\newtheorem{assumption}{Assumption}

\newtheorem{remark}{Remark}

\newcommand\scalemath[2]{\scalebox{#1}{\mbox{\ensuremath{\displaystyle #2}}}}
\newcommand{\tr}{\mathrm{tr}}
\newcommand{\T}{^{\top}}

\newcommand{\Ad}[1]{\mathrm{Ad}_{#1}}

\newcommand{\se}{\mathfrak{se}}
\newcommand{\proj}{\mathrm{Proj}}
\newcommand{\ie}{\emph{i.e.}}
\newcommand{\dom}[1]{\text{dom} #1}
\providecommand{\keywords}[1]{\small \textbf{{Index terms---}}#1}

\begin{document}
%

\title{On the Design of Hybrid Pose and Velocity-bias Observers on Lie Group $SE(3)$}

\author{Miaomiao~Wang~and~Abdelhamid~Tayebi 
	\thanks{
		This work was supported by the National Sciences and Engineering Research Council of Canada (NSERC). A preliminary and partial version of this work was presented in \cite{wang2017globally}.}
	\thanks{
		The authors are with the Department of Electrical and Computer Engineering, University of Western Ontario, London, Ontario, Canada. A. Tayebi is also with the Department of Electrical Engineering, Lakehead University, Thunder Bay, Ontario, Canada.
		{\tt\small mwang448@uwo.ca}, {\tt\small atayebi@lakeheadu.ca}}%
}


\maketitle

\begin{abstract}
This paper deals with the design of globally exponentially stable invariant observers on the Special Euclidian group $SE(3)$. First, we propose a generic hybrid observer scheme (depending on a generic potential function) evolving on {$SE(3)\times \mathbb{R}^6$} for pose (orientation and position) and velocity-bias estimation. Thereafter, the proposed observer is formulated explicitly in terms of inertial vectors and landmark measurements. Interestingly, the proposed observer leads to a decoupled rotational error dynamics from the translational dynamics, which is an interesting feature in practical applications with noisy measurements and disturbances.
\end{abstract}

\keywords{Nonlinear Observer, Special Euclidean group $SE(3)$, Hybrid dynamics, Inertial-vision systems.}

\section{Introduction}
The development of reliable pose (\textit{i.e,} attitude and position) estimation algorithms is instrumental in many applications such as  autonomous underwater vehicles and unmanned aerial vehicles.
Since there is no sensor that directly measures the attitude, the latter is usually determined using body-frame measurements of some known inertial vectors via static determination algorithms \cite{Shuster1981} which are generally sensitive to measurement noise. Alternatively, dynamic estimation algorithms using inertial vector measurements together with the angular velocity can be used to recover the attitude while filtering measurement noise (\textit{e.g.,} Kalman filters \cite{crassidis2007survey}, linear complementary filters \cite{tayebi2006attitude}, nonlinear complementary filters \cite{mahony2008nonlinear}). In low-cost applications, angular velocity and inertial vector measurements can be obtained, for instance, from an inertial measurement unit (IMU) equipped with gyroscopes, accelerometers and magnetometers. The translational position and velocity can be estimated using a Global Positioning System (GPS). However, in GPS-denied environments such as indoor applications, recovering the position and linear velocity is a challenging task. Alternatively, inertial-vision systems combining IMU and on-board camera measurements have been considered for pose estimation \cite{rehbinder2003pose,tarek2017riccati,moeini2016global}. In \cite{tarek2017riccati}, local Riccati-based pose observers have been proposed relying on the system's linear and angular velocities and the bearing measurements of some known landmark points. Another solution with global asymptotic stability (GAS) has been proposed in\cite{moeini2016global}, which considers a non-geometric pose estimation problem using biased body-frame measurements of the system's linear and angular velocities as well as body-frame measurements of landmarks. The  achieved global asymptotic stability results are due to the fact that the estimates are not confined to live in $SE(3)$ for all times.  

Recently, a class of nonlinear observers on Lie groups including $SO(3)$ and $SE(3)$ have made their appearances in the literature. Invariant observers which take into account the topological properties of the motion space were developed in \cite{mahony2008nonlinear,bonnabel2009non,lageman2010gradient}. Motivated by the work of \cite{mahony2008nonlinear} on $SO(3)$, complementary observers on $SE(3)$ were proposed in \cite{baldwin2007complementary,hua2011observer}. In practice, measurements of group velocity (translational and rotational velocities) are often corrupted by an unknown bias. Pose estimation using biased velocity measurements were considered in  \cite{vasconcelos2010nonlinear,khosravian2015observers,hua2015gradient}. A nice feature of  \cite{hua2015gradient} is the fact that the observer incorporates (naturally) both inertial vector measurements (\textit{e.g.,} from IMU) and landmark measurements (\textit{e.g.,} from a vision system). The observers proposed in \cite{baldwin2007complementary,hua2011observer,vasconcelos2010nonlinear,khosravian2015observers,hua2015gradient} are shown to guarantee almost global asymptotic stability (AGAS), \textit{i.e.,} the estimated pose converges to the actual one from almost all initial conditions except from a set of Lebesgue measure zero. This is the strongest result one can aim at when considering continuous time-invariant state observers on $SO(3)$ or $SE(3)$.   

Recently, the topological obstruction to global asymptotic stability on $SO(3)$ using continuous time-invariant controllers (observers) has been successfully addressed via hybrid techniques such as \cite{mayhew2011hybrid,wu2015globally,berkane2017construction,berkane2017CDC,berkane2017hybrid2}. To this end, a synergistic hybrid technique was introduced in \cite{mayhew2011hybrid}. Motivated by this approach, globally asymptotically stable hybrid attitude observers on $SO(3)$ have been proposed in \cite{wu2015globally}  and globally exponentially stable hybrid attitude observers on $SO(3)$ have been proposed in \cite{berkane2017hybrid2}. 

In this paper we propose an generic approach for hybrid observers design on $SE(3)$ leading to global exponential stability. To the best of our knowledge, there is no work in the literature achieving such results on $SE(3)$.  Moreover, we propose some explicit and practically implementable versions of the proposed scheme, using biased group-velocity measurements (constant bias), inertial vectors and landmark measurements. Interestingly, the proposed observer leads to a decoupled rotational error dynamics from the translational error dynamics, which guarantees nice robustness and performance properties. 

The remainder of this paper is organized as follows. Section \ref{sec:backgroud} introduces some preliminary notions that will be used throughout out the paper. In Section \ref{sec:problem_statemente}, the problem of pose estimation on $SE(3)$ is formulated. In Section \ref{sec:global_exponential}, a hybrid approach for pose and velocity-bias estimation is proposed. An estimation scheme leading to decoupled rotational error dynamics from the translational error dynamics is provided in Section \ref{sec: fully_decoupled}. Section \ref{sec:simulation} presents some simulation results showing the performance of the proposed estimation schemes.

\section{Background and Preliminaries} \label{sec:backgroud}
\subsection{Notations and mathematical preliminaries}
The sets of real, nonnegative real and natural number are denoted as $\mathbb{R}$, $\mathbb{R}_{\geq 0}$ and $\mathbb{N}$, respectively. We denote by $\mathbb{R}^n$ the $n$-dimensional Euclidean space and $\mathbb{S}^n$ the set of $n$-dimensional unit vectors. Given two matrices, $A,B\in \mathbb{R}^{m\times n}$, their Euclidean inner product is defined as $\langle\langle A,B\rangle\rangle := \tr(A\T B)$. The Euclidean norm of a vector $x\in \mathbb{R}^n$ is defined as $\|x\| := \sqrt{x\T x}$, and the Frobenius norm of a matrix $X\in \mathbb{R}^{n\times m}$ is given by $\|X\|_F := \sqrt{\langle \langle X, X\rangle\rangle}$. The $n$-by-$n$ identity matrix is denoted by $I_n$. For a matrix $A\in \mathbb{R}^{n\times n}$, we define $\mathcal{E}(A)$ as the set of all eigenvectors of $A$ and $\mathbb{E}(A)\subseteq \mathcal{E}(A)$ as the eigenbasis set of $A$. Let $\lambda^A_i$ be the $i$-th eigenvalue of $A$, and $\lambda^A_{\min}$ and $\lambda^A_{\max}$ be the minimum and maximum eigenvalue of $A$, respectively.

Let $SO(3): = \{R\in \mathbb{R}^{3\times 3}| R\T R=I_3, \det(R)=1\}$ be the 3-dimensional \textit{Special Orthogonal group} $SO(3)$. Consider the $3$-dimensional \textit{Special Euclidean group}, defined as
\[
SE(3): = \left\{g=\mathcal{T}(R,p) \in  \mathbb{R}^{4\times 4} \left|  
R\in SO(3),p\in \mathbb{R}^3\right. \right\},
\]
with the map $\mathcal{T}:SO(3)\times \mathbb{R}^3 \to SE(3)$ given by
\[
\mathcal{T}(R,p) :=\begin{bmatrix}
R & p \\
0  & 1
\end{bmatrix}.
\]
On any Lie group the tangent space at the group identity has the structure of a \textit{Lie algebra}.
Let $\mathfrak{so}(3):=\{\Omega\in \mathbb{R}^{3\times 3}| \Omega\T=-\Omega\}$ be the Lie algebra of $SO(3)$. The Lie algebra of  $SE(3)$, denoted by $\mathfrak{se}(3)$, is given by
\begin{align*}
\mathfrak{se}(3) &:= \left\{X \in \mathbb{R}^{4\times 4}| X= \begin{bmatrix}
\Omega & v \\
0  & 0
\end{bmatrix}, \Omega \in \mathfrak{so}(3), v\in \mathbb{R}^3\right\}.
\end{align*}
Let $\times$ be the vector cross-product on $\mathbb{R}^3$ and define the map $(\cdot)^\times: \mathbb{R}^3 \mapsto \mathfrak{so}(3)$ such that $x \times y = x^\times y$, for any $x,y\in \mathbb{R}^3$.
A wedge map $(\cdot)^\wedge: \mathbb{R}^6 \mapsto \mathfrak{se}(3)$ is defined as
\[
\xi^\wedge := \begin{bmatrix}
\omega^\times & v \\
0 & 0
\end{bmatrix},\quad  \xi := \begin{bmatrix}
\omega \\
v
\end{bmatrix}.
\]
The \textit{tangent space} of the group $SE(3)$, is identified by
$
T_g SE(3) : = \{gX |\ g\in SE(3), X\in \mathfrak{se}(3)\}.
$
Let $\langle\cdot,\cdot\rangle_g: T_g SE(3) \times T_g SE(3) \to \mathbb{R}$ be a \textit{Riemannian metric} on $SE(3)$, such that
\begin{align*}
\langle gX,gY\rangle_g := \langle\langle X,Y\rangle\rangle, \quad g\in SE(3), X,Y \in \mathfrak{se}(3).
\end{align*}
Given a differentiable smooth function $f: SE(3) \to \mathbb{R}$, the $\textit{gradient}$ of $f$ , denoted $\nabla_g f \in T_g SE(3) $, relative to the Riemannian metric $\langle\cdot,\cdot\rangle_g$ is uniquely defined by
$$
df\cdot gX = \langle \nabla_g f ,gX\rangle_g = \langle\langle g^{-1}\nabla_g f ,X\rangle\rangle,  
$$
For all $g\in SE(3), X \in \mathfrak{se}(3)$. A point $g\in SE(3)$ is called \textit{critical point} of $f$ if the gradient of $f$ at $g$ is zero (\textit{i.e.}, $\nabla_g f =0$). The set of all critical points of $f$ on $SE(3)$ is denoted by $\mathcal{C} f := \{g \in SE(3) | \nabla_g f =0\} \subseteq SE(3)$. For any $g\in SE(3)$, we define $|g|_I$ as the distance with respect to $I_4$, which is given by $|g|_I := \|I_4-g\|_F$. The map $\mathcal{R}_a :\mathbb{R}\times \mathbb{S}^2 \to SO(3)$ known as the angle-axis parametrization of $SO(3)$ is defined as
$
\mathcal{R}_a(\theta,u) := I_3 + \sin\theta u^\times + (1-\cos\theta) (u^\times)^2.
$

For a matrix $A\in \mathbb{R}^{3\times 3}$, we denote by $\mathbb{P}_a(A)$ as the anti-symmetric projection of A, such that $\mathbb{P}_a(A) = (A-A\T)/2$. Let  $\mathbb{P}:\mathbb{R}^{4\times 4} \rightarrow \mathfrak{se}(3)$ denote the projection of $\mathbb{A}$ on the Lie algebra  $\mathfrak{se}(3)$, such that, for all $X\in \mathfrak{se}(3)$, $ \mathbb{A} \in \mathbb{R}^{4\times 4}$ one has
$ \langle\langle\mathbb{A}, X\rangle\rangle=\langle\langle X,\mathbb{P}( \mathbb{A} )\rangle\rangle = \langle\langle \mathbb{P}( \mathbb{A} ),X\rangle\rangle.
$
For all $A\in \mathbb{R}^{3\times 3}, b,c\T \in
\mathbb{R}^3$ and $d\in \mathbb{R}$, one has
\begin{equation}
\mathbb{P}\left(\begin{bmatrix}
A & b \\
c & d
\end{bmatrix}\right) := \begin{bmatrix}
\mathbb{P}_a(A) & b\\
0_{1 \times 3} & 0
\end{bmatrix}.
\end{equation}
Let $\text{vex}: \se(3) \to \mathbb{R}^6$ denote the inverse isomorphism of the map $(\cdot)^\wedge$, such that $\text{vex}(\xi^\wedge) = \xi$ and $(\text{vex}(X))^\wedge = X$, for all $\xi\in \mathbb{R}^6$ and $X\in \se(3)$. For a matrix $\mathbb{A} = \begin{bmatrix}
A & b\\
c & d
\end{bmatrix} $, we also define the following maps
\begin{equation}
\psi_b(\mathbb{A} ):= \text{vex} (\mathbb{P}(\mathbb{A}))=  \begin{bmatrix}
\psi_a(A) \\
b
\end{bmatrix},  \psi_a(A) := \frac{1}{2}\begin{bmatrix}
a_{32} - a_{23} \\
a_{13} - a_{31} \\
a_{21} - a_{12}
\end{bmatrix} \label{eqn:definiton_psi},
\end{equation}
with $A=[a_{ij}]$. 
It is verified that for all $\mathbb{A}\in \mathbb{R}^{4\times 4}, y\in \mathbb{R}^6$,
\begin{align*}
\langle\langle \mathbb{A}, y^\wedge  \rangle\rangle = 2\psi_b(\mathbb{A})\T \Theta y \quad \text{with }  \Theta :=  \text{diag}(\begin{bmatrix}
 I_3  & \frac{1}{2}I_3
\end{bmatrix} ) 
\end{align*}


Given a rigid body with configuration $g\in SE(3)$, the \textit{adjoint map} $\Ad{g}(\cdot): SE(3) \times \se(3) \to \se(3)$ is given by
$
\Ad{g}(X)   := g X  g^{-1}   \text{~with } g\in SE(3), X \in \se(3).
$
The matrix representation of the adjoint map on $\se(3)$ is defined as
\begin{equation}
\Ad{g}:= \begin{bmatrix}
R & 0\\
p^\times R & R
\end{bmatrix} \in \mathbb{R}^{6\times 6},
\end{equation}
such that $g x^\wedge g^{-1} = (\Ad{g}x )^\wedge$, for all $g\in SE(3), x\in \mathbb{R}^6$. One verifies that $\Ad{g_1} \Ad{g_2} = \Ad{g_1 g_2}$, for all $g_1, g_2\in SE(3)$. Define $\Ad{g}^*(\cdot)$ as the Hermitian adjoint of $\Ad{g}(\cdot)$ with respect to the inner product $\langle\langle\cdot, \cdot\rangle\rangle$ on $\se(3)$ associated with the right-invariant Riemannian metric, such that
$
 \langle \langle Y, \Ad{g} (X) \rangle\rangle = \langle  \langle \Ad{g}^* (Y), X \rangle \rangle,
$
 for all $X, Y\in \se(3), g\in SE(3)$.  
The matrix representation of the map $\Ad{g}^*(\cdot)$ is given by
\begin{align}
\Ad{g}^* := \begin{bmatrix}
R\T & -R\T p^\times \\
0 & R\T
\end{bmatrix}. 
\end{align}
For all $X\in \se(3)$ and $g\in SE(3)$, one has $\psi_b(g\T X g^{-\top}) = \Theta^{-1}\Ad{g}^*\Theta\psi_b(X)$. 
For the sake of simplicity, let us define the map $\psi:\mathbb{R}^{4\times 4} \to \mathbb{R}^6$ such that $\psi(\mathbb{A})= \Theta \psi_b(\mathbb{A})$, for all $\mathbb{A}\in \mathbb{R}^{4\times 4}$. The following identities are used throughout this paper:
\begin{align}
&\langle\langle \mathbb{A}, y^\wedge  \rangle\rangle  = 2\psi(\mathbb{A})\T   y  \label{eqn:psi_property}. \\
&\psi(g\T X g^{-\top})  = \Ad{g}^*\psi(X)
\end{align}

For any two vectors $r ,b \in \mathbb{R}^4$, we define the following wedge product (exterior product)  $\wedge$ as
\begin{equation}
b \wedge r :
= \begin{bmatrix}
b_v\times r_v \\
b_s r_v -
r_s b_v
\end{bmatrix} \in \mathbb{R}^6, \label{eqn:wedge_product}
\end{equation}
where $r = (r_{v}\T,r_{s})\T, b = (b_{v}\T,b_{s})\T$ with $r_v, b_v \in \mathbb{R}^3$ and $r_s, b_s \in \mathbb{R}$. For all $r,b \in \mathbb{R}^4, g\in SE(3)$, one can easily verify that $ r \wedge r=0, b\wedge r= - r\wedge b $, and
\begin{align} 
&((g b) \wedge (g r))^\wedge 
=(\Ad{g^{-1}}^* (b\wedge r))^\wedge = g^{-\top} (b\wedge r)^\wedge g^{-1} \label{eqn:property_1}\\
&\psi ((I_4 -g)r r\T  )  =  \frac{1}{2}(gr) \wedge r \label{eqn:property_2}.
\end{align}
Let $\mathcal{M}_0$ denote the sub-manifold of $\mathbb{R}^{4\times 4}$, defined as
	\[
	\mathcal{M}_0:= \left\{M  \left| M=\begin{bmatrix}
	M_1 & m_2 \\
	0 & 0
	\end{bmatrix},M_1 \in \mathbb{R}^{3\times 3}, m_2 \in \mathbb{R}^3 \right. \right\}.
	\]
	Then, for all $g\in SE(3), M, \bar{M}\in \mathcal{M}_0$, the following properties hold
	\begin{align}
	&\mathbb{P}(gM)  = \mathbb{P}(g^{-\top} M),   \label{eqn:property1} \\
	& \tr(g\T g M \bar{M}\T)    =  \tr(M\bar{M}\T). \label{eqn:property2}
	\end{align} 
Throughout the paper, we will also make use of the following matrix decomposition:
\begin{equation}
\begin{bmatrix}
A &b \\
b\T & d
\end{bmatrix}=\begin{bmatrix}
I_3 &bd^{-1} \\
0 & 1
\end{bmatrix} \begin{bmatrix}
A - bd^{-1}b\T & 0 \\
0 & d
\end{bmatrix}\begin{bmatrix}
I_3 & 0 \\
b\T d^{-1} & 1
\end{bmatrix}.  \label{eqn:A_dcomposition}
\end{equation}

\subsection{Hybrid Systems Framework}
Define a \textit{hybrid time domain} as a subset $E \subset \mathbb{R}_{\geq 0} \times \mathbb{N}$ in the form
\[
E = \bigcup_{j=0}^{J-1} ([t_j,t_{j+1}] \times \{j\})
\]
for some finite sequence $0=t_0 \leq t_1 \leq \cdots \leq t_J$, with the ``last'' interval possibly of the form $([t_{J-1}, t_J)\times J)$ or $([t_{J-1}, +\infty)\times J)$. On each hybrid time domain there is a natural ordering of points : $(t,j)\preceq (t',j')$ if $t\leq t'$ and $j\leq j'$.

Given a manifold $\mathcal{M}$, we consider the following hybrid system \cite{goebel2009hybrid}:
\begin{align}
\mathcal{H}: ~ \begin{cases}
\dot{x}~~ \in F(x),& \quad x \in \mathcal{F}   \\
x^{+} \in G(x),& \quad x \in \mathcal{J}
\end{cases}
\end{align}
where the \textit{flow map} $F: \mathcal{M} \to T\mathcal{M}$ describes the continuous flow of $x$ on the \textit{flow set} $\mathcal{F} \subset \mathcal{M}$; the \textit{jump map} $G: \mathcal{M} \to T\mathcal{M}$ describes the discrete flow of $x$ on the \textit{jump set} $\mathcal{J} \subset \mathcal{M}$. A hybrid arc is a function $x: \dom{x} \to \mathbb{R}^n$, where $\dom{x}$ is a hybrid time domain and, for each fixed $j$, $t \mapsto x(t,j)$ is a locally absolutely continuous function on the interval $I_j = \{t:(t,j) \in \dom{x}\}$. For more details on dynamic hybrid systems, we refer the the reader to \cite{goebel2009hybrid,goebel2012hybrid} and references therein. 


\section{Problem Formulation and Preliminary Results} \label{sec:problem_statemente}
Let $\mathcal{I}$ be an inertial frame and $\mathcal{B}$ be a body-fixed frame. Let $p\in \mathbb{R}^3$ denote the rigid body position expressed in the inertial frame $\mathcal{I}$, and $R\in SO(3)$ the rigid body attitude describing the rotation of frame $\mathcal{B}$ with respect to frame $\mathcal{I}$.
We consider the problem of pose estimation of the rigid body, \textit{i.e.,} position $p$ and attitude $R$. The pose of the rigid body can be represented by
$
g  =\mathcal{T}(R,p)\in SE(3).
$
This representation is commonly known as the homogeneous representation. Let $\omega\in \mathbb{R}^3$ denote the angular velocity of the body-fixed frame $\mathcal{B}$ with respect to the inertial frame $\mathcal{I}$, expressed in frame $\mathcal{B}$. Let $v \in \mathbb{R}^3$ be the translational velocity, expressed in frame $\mathcal{B}$. The pose is governed by the following dynamics:
\begin{equation}
\dot{g} = g \xi^\wedge,  \label{eqn:kenimatic_g}
\end{equation}
where $\xi := (\omega\T,v\T)\T \in \mathbb{R}^6$.
Note that system (\ref{eqn:kenimatic_g}) is left invariant in the sense that it preserves the Lie group invariance properties with respect to constant translation and constant rotation of the body-fixed frame $\mathcal{B}$. Let the group velocity be piecewise-continuous, and consider the following biased group velocity measurement:
\begin{equation}
\xi_y := \xi + b_a . \label{eqn:velocity_bias}
\end{equation}
where $b_a := (b_\omega\T,b_v\T)\T \in \mathbb{R}^6$ with $b_\omega,b_v\in \mathbb{R}^3$ denotes the unknown constant velocity biases.
Moreover, a family of $n$ constant homogeneous vectors $r_i \in \mathbb{R}^4, i=1,2,\cdots,n$, known in the inertial frame $\mathcal{I}$, are assumed to be measured in the  frame $\mathcal{B}$ as
\begin{equation}
b_i =h(g,r_i):= g^{-1}r_i, \quad i=1,2,\cdots,n.  \label{eqn:output_measurement}
\end{equation}
Assume that, among the $n$ inertial elements, there are $n_1$ feature points (or landmarks) $p_i^\mathcal{I}, i=1,2,\cdots, n_1$, and $n_2=n-n_1$ inertial vectors $v_i^\mathcal{I}, i=1,2,\cdots,n_2$, \ie,
\begin{equation}
r_i = \begin{bmatrix}
p_i^\mathcal{I} \\
1
\end{bmatrix},  i=1,\cdots,n_1 \quad
r_{j+n_1} = \begin{bmatrix}
v_{j}^\mathcal{I} \\
0
\end{bmatrix},  j=1,,\cdots,n_2
\label{eqn:definition_r_i}
\end{equation}

Define the following modified inertial vectors and weighted geometric landmark center:
\begin{equation}
\textstyle \bar{v}_i^\mathcal{I} := p_i^{\mathcal{I}}  - p_c^\mathcal{I},\quad p^{\mathcal{I}}_c  :=\sum_{i=1}^{n_1} \alpha_i p^{\mathcal{I}}_i
\label{eqn:p_c}
\end{equation}
with $\alpha_i > 0$ for all $i=1,2,\cdots,n_1$ and $\sum_{i=1}^{n_1} \alpha_i =1$.

\begin{assumption}\label{assump:1}
	Among the $n$ measurements, at least one landmark point is measured, and at least two vectors from the set $V^{\mathcal{I}} := \{\bar{v}_1^\mathcal{I},\cdots,\bar{v}_{n_1}^\mathcal{I},v_1^\mathcal{I}, \cdots, v_{n-n_1}^\mathcal{I}\}$ are non-collinear.
\end{assumption}
\begin{remark}	
	From Assumption \ref{assump:1} one verifies that $n_1\geq 1$ and $n\geq 3$. This assumption is standard in estimation problems in $SE(3)$, \textit{e.g.,} \cite{vasconcelos2010nonlinear,hua2011observer,hua2015gradient}, which is satisfied in the following particular cases:
	\begin{itemize}
		\item Three different landmark points are measured such that the corresponding $\bar{v}_i^{\mathcal{I}}$, $i=1,2,3$, are non-collinear.
		\item One landmark point and two non-collinear inertial vectors are measured.
		\item Two different landmark points and one inertial vector are measured such that the corresponding $v_1^\mathcal{I}$ and $\bar{v}_i^\mathcal{I},i=1,2$ are non-collinear.
	\end{itemize}
\end{remark}

\begin{assumption}\label{assump:2}
	The pose $g$ and group velocity $\xi$ of the rigid body are uniformly bounded.
\end{assumption}

Our objective is to design a globally exponentially stable hybrid pose and velocity-bias estimation scheme that provides estimates $\hat{g}$ and $\hat{b}_a$ of $g$ and $b_a$, respectively, using the available measurements satisfying Assumption \ref{assump:1} and Assumption \ref{assump:2}.

\section{Gradient-based Hybrid Observer Design} \label{sec:global_exponential}
In this paper, we make use of the framework of hybrid dynamical systems presented in\cite{goebel2009hybrid,goebel2012hybrid}.
Consider a positive-valued continuously differentiable function $\mathcal{U}:SE(3) \to \mathbb{R}_{\geq 0}$. The function $\mathcal{U}(g)$ is said to be a potential function on $SE(3)$ if $\mathcal{U}(g)\geq 0$ for all $g\in SE(3)$ and $\mathcal{U}(g)=0$ if and only if $g=I_4$. For all $g\in SE(3)$, $\nabla_g \mathcal{U}(g)$ denotes the gradient of $\mathcal{U}$ with respect to $g$. Let $ \Psi_{\mathcal{U}}$ denote the set of critical points and $\mathcal{X}_{\mathcal{U}} := \Psi_{\mathcal{U}}/\{I_4\}$ be the set of undesired critical points \footnote{As shown in \cite{koditschek1989application}, no smooth vector field on Lie groups, which are not homeomorphic to $\mathbb{R}^n$, can have a global attractor. Therefore, any smooth potential function on $SE(3)$(or $SO(3)$), has at least four critical points.}.

\subsection{Generic hybrid pose and velocity-bias estimation filter}
Let $\hat{g}:=\mathcal{T}(\hat{R},\hat{p})$ and $\hat{b}_a$ denote, respectively, the estimates of the rigid body pose and velocity bias. Define the pose estimation error $\tilde{g}:=\mathcal{T}(\tilde{R},\tilde{p})=g\hat{g}^{-1}$ and bias estimation error $\tilde{b}_a := \hat{b}_a - b_a$. Given a nonempty finite set $\mathbb{Q}\subset SE(3)$, we propose the following generic hybrid pose and velocity-bias estimation scheme relying on a generic potential function on $SE(3)$:
\begin{align}
&\underbrace{\begin{array}{ll}
	\dot{\hat{g}}~~  = \hat{g} (\xi_y - \hat{b}_a+ k_{\beta}\beta )^\wedge  	\\
	\dot{\hat{b}}_a~  =   -\Gamma \sigma_b
	\end{array}} _{(\hat{g},\hat{b}_a)~\in~ \mathcal{F}_o }
\underbrace{\begin{array}{ll}
	\hat{g}^{+} ~ = g_q^{-1} \hat{g},  g_q \in \gamma({\hat{g}})      \\[0.1cm]
	\hat{b}_a^{+} ~ = \hat{b}_a
	\end{array}  } _{(\hat{g},\hat{b}_a)~\in~ \mathcal{J}_o } ,  \label{eqn:observer_design} \\
&~~~\beta  ~~:= \Ad{\hat{g}^{-1}} \psi({\tilde{g}}^{-1}\nabla_{{\tilde{g}}} \mathcal{U}({\tilde{g}}) ) , \label{eqn:observer_design_beta}\\
&~~~\sigma_b  ~:=   \Ad{{\hat{g}}}^* \psi({\tilde{g}}^{-1}\nabla_{{\tilde{g}}} \mathcal{U}({\tilde{g}},)) ,\label{eqn:observer_design_sigma}
\end{align}
where $\hat{g}(0) \in SE(3),  \hat{b}_a(0) \in \mathbb{R}^6$, $\Gamma := \text{diag}(k_{\omega} I_3, k_v I_3) \in \mathbb{R}^{6\times 6}$ and $k_\omega,k_v, k_{\beta}> 0$. The set-valued map $\gamma: SE(3) \rightrightarrows \mathbb{Q}$ is defined as $\gamma(\hat{g}) := \{g_q\in \mathbb{Q}: g_q=\arg \min_{g_q\in \mathbb{Q}} \mathcal{U}(\tilde{g}g_q) \}$.
The flow set $\mathcal{F}_o$ and jump set $\mathcal{J}_o$ are defined by
\begin{align}
\mathcal{F}_o: =  \{(\hat{g},\hat{b}_a)\in SE(3) \times \mathbb{R}^6 : \mathcal{U}(\tilde{g}) - \min_{g_q\in \mathbb{Q}}\mathcal{U}(\tilde{g}g_q) \leq \delta  \},   \label{eqn: definition_flow_set}\\
\mathcal{J}_o: =  \{(\hat{g},\hat{b}_a)\in SE(3)\times \mathbb{R}^6: \mathcal{U}(\tilde{g}) - \min_{g_q\in \mathbb{Q}}\mathcal{U}(\tilde{g}g_q) \geq \delta  \} ,  \label{eqn: definition_jump_set}
\end{align}
for some $\delta > 0$. The potential function $\mathcal{U}$, and the parameters $\delta$ and $g_q\in \mathbb{Q} \subset SE(3)$ will be designed later. Note that the vector $\xi_y$ involved in (\ref{eqn:observer_design}) is a known bounded function of time. Note also that $\mathcal{U}(\tilde{g})$ and $\mathcal{U}(\tilde{g}g_q)$ involved in (\ref{eqn: definition_flow_set}) and (\ref{eqn: definition_jump_set}), and $\psi(\tilde{g}^{-1}\nabla_{\tilde{g}}\mathcal{U}(\tilde{g}))$ involved in (\ref{eqn:observer_design_beta}) and (\ref{eqn:observer_design_sigma}) can be rewritten in terms of $\hat{g}$ and the available measurements as it is going to be shown later.
We define the extended space and state as
$
\mathcal{S}:= SE(3)\times \mathbb{R}^6 \times SE(3)\times \mathbb{R}^6 \times \mathbb{R}_{\geq 0} \text{ and }
x:= (\tilde{g},\tilde{b}_a,\hat{g},\hat{b}_a,t) \in \mathcal{S}.
$
In view of (\ref{eqn:kenimatic_g}), (\ref{eqn:velocity_bias}), (\ref{eqn:observer_design})-(\ref{eqn:observer_design_sigma}), one has the following hybrid closed-loop system:
\begin{align}
\mathcal{H}: \begin{cases}
\dot{x}~~\in F(x) & x\in \mathcal{F}_c:=\{x\in \mathcal{S}: (\hat{g},\hat{b}_a)\in \mathcal{F}_o\} \\
x^+\in G(x) & x\in  \mathcal{J}_c :=\{x\in \mathcal{S}: (\hat{g},\hat{b}_a)\in \mathcal{J}_o\}
\end{cases}   \label{eqn:closed_loop}
\end{align}
with
\begin{align}
\scalemath{0.95}{F(x)= \begin{bmatrix}
	{\tilde{g}}   (\Ad{{\hat{g}}}\tilde{b}_a - k_\beta  \psi({\tilde{g}}^{-1}\nabla_{{\tilde{g}}}\mathcal{U}({\tilde{g}}))   )^\wedge \\
	-\Gamma \Ad{{\hat{g}}}^* \psi({\tilde{g}}^{-1}\nabla_{{\tilde{g}}} \mathcal{U}({\tilde{g}}) )\\
	\hat{g} (\xi_y - \hat{b}_a+k_\beta \Ad{\hat{g}^{-1}}  \psi({\tilde{g}}^{-1}\nabla_{{\tilde{g}}} \mathcal{U}({\tilde{g}}) )  )^\wedge   \\
	-\Gamma \Ad{{\hat{g}}}^* \psi({\tilde{g}}^{-1}\nabla_{{\tilde{g}}} \mathcal{U}({\tilde{g}}) ) \\
	1
	\end{bmatrix} , G(x)= \begin{bmatrix}
	{\tilde{g}}g_q\\
	\tilde{b}_a \\
	g_q^{-1}\hat{g} \\
	\hat{b}_a \\
	t
	\end{bmatrix}}. \nonumber
\end{align}

Note that the closed-loop system (\ref{eqn:closed_loop}) satisfies the hybrid basic conditions of \cite{goebel2009hybrid} and is autonomous.

Define the closed set $\bar{\mathcal{A}}  := \{x\in   \mathcal{S}: {\tilde{g}}=I_4, \tilde{b}_a=0\}$ and let $|x|_{\bar{\mathcal{A}}}$ denote the distance to the  set $\bar{\mathcal{A}}$ such that
$
|x|_{\bar{\mathcal{A}}}^2 := \inf_{y=(I_4,0,\bar{g},\bar{b}_a,\bar{t}) \in \bar{\mathcal{A}}} (\|I_4-\tilde{g}\|_F^2 + \|\tilde{b}_a\|^2 + \|\bar{g}-\hat{g}\|_F^2 + \|\bar{b}_a-\hat{b}_a\|^2+ \|\bar{t}-t\|^2) = |\tilde{g}|_I^2 + \|\tilde{b}_a\|^2$. Now, one can state one of our main results.

\begin{theorem}\label{theo:Theorem_1}
	Consider system (\ref{eqn:closed_loop}) with a continuously differentiable potential function $\mathcal{U}$ on $ SE(3)$, and choose the nonempty finite $\mathbb{Q} $ and the gap $\delta >0 $ such that:
	\begin{align}
	&\alpha_1 |{\tilde{g}}|_I^2 \leq \mathcal{U}({\tilde{g}}) \leq \alpha_2 |{\tilde{g}}|_I^2, ~~~~~~~~~~~~~~~~~~~ x\in \mathcal{S}, \label{eq1_theo1}\\
	&\alpha_3 |{\tilde{g}}|_I^2 \leq \|\psi({\tilde{g}}^{-1}\nabla_{{\tilde{g}}} \mathcal{U}({\tilde{g}}) )\|^2 \leq \alpha_4 |{\tilde{g}}|_I^2,  ~~x\in \mathcal{F}_c,\label{eq2_theo1}\\
	& \|\Ad{{\tilde{g}}^{-1}}^* \psi({\tilde{g}}^{-1}\nabla_{{\tilde{g}}} \mathcal{U}({\tilde{g}}) )\|^2 \leq \alpha_5|{\tilde{g}}|_I^2, ~~~~~x\in \mathcal{F}_c, \label{eq3_theo1}
	\end{align}
	where $\alpha_1,\cdots, \alpha_5$ are strictly positive scalars. Let Assumption \ref{assump:1} and Assumption \ref{assump:2} hold. Then, the number of jumps is finite and for any initial condition $x(0,0)\in \mathcal{S}$, the solution $x(t,j)$ is complete and there exist $k>0$ and $\lambda>0$ such that
	\begin{equation}
	|x(t,j)|_{\bar{\mathcal{A}}}\leq k \exp(-\lambda (t+j)) |x(0,0)|_{\bar{\mathcal{A}}},
	\end{equation}
	for all $(t,j)\in \dom{x}$.
\end{theorem}

\begin{proof}
	See Appendix \ref{sec:theo1}.
\end{proof}

\begin{remark}
	Theorem 1, provides exponential stability results for the generic estimation scheme \eqref{eqn:observer_design}-\eqref{eqn: definition_jump_set} relying on a generic potential function $\mathcal{U}$. The flow and jump sets $\mathcal{F}_o$ and $\mathcal{J}_o$, given in \eqref{eqn: definition_flow_set}-\eqref{eqn: definition_jump_set}, depend on some generic parameters $\delta$ and $g_q$ that have to be designed together with the potential function $\mathcal{U}$ such that conditions \eqref{eq1_theo1}-\eqref{eq3_theo1} are fulfilled. It is worth pointing out that condition \eqref{eq2_theo1} implies that the undesired critical points belong to the jump set $\mathcal{J}_o$. In the next section, we will design $\mathcal{U}$, $\delta$ and $\mathbb{Q}$ such that \eqref{eq1_theo1}-\eqref{eq3_theo1} are fulfilled.
\end{remark}

\begin{remark}
	The filters proposed in \cite{baldwin2007complementary,hua2011observer,vasconcelos2010nonlinear,khosravian2015observers,hua2015gradient} are shown to guarantee almost global asymptotic stability due to the topological obstruction when considering continuous time-invariant state observers on $SE(3)$. The proposed hybrid estimation scheme, uses a new observer-state jump mechanism, inspired from \cite{berkane2017CDC}, which changes directly the observer state through appropriate jumps in the direction of a decreasing potential function on $SE(3)$. The jump transitions occur when the estimation error is close to the critical points. This observer-state jump mechanism is different from the principle used, for instance in \cite{mayhew2011hybrid,wu2015globally,berkane2017hybrid2}, which consists in incorporating the jumps in the observer's correcting term derived from a family of synergistic potential functions.
\end{remark}

\subsection{Explicit hybrid observers design using the available measurements}
In this subsection, we provide an explicit expression for the hybrid observers in terms of available measurements and estimated states. Before proceeding with the design, some useful properties are given in the following lemmas (whose proof are given in appendix).

\begin{lemma} \label{lemma:Q}
	Consider a family of $n$ elements of homogeneous space $r_i \in \mathbb{R}^4, i=1,2,\cdots,n$ defined in (\ref{eqn:definition_r_i}). Given $k_i\geq 0$ for all $i=1,2,\cdots,n$, define the following matrix
	\begin{equation}
	\mathbb{A}: = \sum_{i=1}^nk_i r_i r_i\T = \begin{bmatrix}
	A & b \\
	b\T & d
	\end{bmatrix}\in \mathbb{R}^{4\times 4},  \label{eqn:definition_A}
	\end{equation}
	where
	$ \textstyle
	A  := \sum_{i=1}^{n_1} k_ip_i^{\mathcal{I}}(p_i^{\mathcal{I}})\T + \sum_{j=1}^{n_2} k_{j+n_1} v_j^{\mathcal{I}}(v_j^{\mathcal{I}})\T  ,
	b  :=  \sum_{i=1}^{n_1} k_ip_i^{\mathcal{I}}  \text{ and }  d:=  \sum_{i=1} ^{n_1}k_i .
	$
	Then, under Assumption \ref{assump:1}  the following statements hold:
	\begin{itemize}
		\item [1)] $d> 0$ .
		\item[2)] Matrix $Q: = A - bb\T d^{-1}$, which can be expressed as
		\begin{align*}
				Q=     \sum_{i=1}^{n_1} k_i \bar{v}_i^{\mathcal{I}}(\bar{v}_i^{\mathcal{I}} )\T +  \sum_{j=1}^{n_2} k_{j+n_1}v_j^{\mathcal{I}}(v_j^{\mathcal{I}})\T ,  
		\end{align*}
		is positive semi-definite.
		\item [3)] Matrix $\bar{Q}:=\frac{1}{2}(\tr(Q)I_3-Q)$, which can be expressed as 
		\begin{align*}
		\bar{Q}=   \frac{1}{2}  \sum_{i=1}^{n_1} k_i  ((\bar{v}_i^{\mathcal{I}})^\times )^2 +  \frac{1}{2}  \sum_{j=1}^{n_2} k_{j+n_1} ((v_j^{\mathcal{I}})^\times)^2,
		\end{align*}
		is positive definite.
	\end{itemize}
\end{lemma}

\begin{lemma} \label{lemma:Delta}
	Let $Q=Q\T$ be a positive semi-definite matrix. Consider the map $\Delta_Q: \mathbb{S}^2 \times \mathbb{S}^2 \to \mathbb{R}$ defined as:
	\begin{equation}
	\Delta_Q(u,v) = u\T((\tr(Q)-2v\T Qv)I_3 - Q+ 2Qvv\T)u \label{eqn:Delta_u_v}.
	\end{equation}
	Let $\mathbb{U}\subset \mathbb{S}^2$ be a finite set of unit vectors.  Define the   constant scalar
	$
	\Delta_Q^*:= \min_{v\in \mathcal{E}(Q)} \max_{u\in \mathbb{U}} \Delta_Q(u,v) .
	$
	Then, the following results hold:	
	\begin{itemize}
		\item [1)] Let $\mathbb{U}$ be a superset of $\mathbb{E}(Q)$ (\textit{i.e.}, $\mathbb{U}\supseteq \mathbb{E}(Q)$), then the following inequality holds:
		\begin{align}
		\Delta_Q^*  \geq \begin{cases}
		\frac{2}{3}  \lambda_1^Q & \textit{if } \lambda_1^Q = \lambda_2^Q = \lambda_3^Q  >0\\
		\min\left\{\lambda^Q_{1}+ \lambda^Q_{2}, \lambda_3^Q \right\} &  \textit{if } \lambda_1^Q = \lambda_2^Q \neq \lambda_3^Q >0 \\
		 \tr(Q) - \lambda_{\max}^Q  &  \textit{if } \lambda_i^Q \neq \lambda_j^Q \geq 0, i\neq j 
		\end{cases} \label{eqb:solution_delta_U}
		\end{align}
		\item [2)] Let $Q$ be a matrix such that $\tr(Q) -2\lambda_{\max}^Q> 0$, and let $\mathbb{U}$ be a set that contains any three orthogonal unit vectors in $\mathbb{R}^3$, then the following inequality holds:
		\begin{equation}
		\Delta_Q^*  \geq  {   \frac{2}{3}} (\tr(Q)  - 2\lambda_{\max}^Q  ) .\label{eqn:Delta_Q2}
		\end{equation}	
	\end{itemize}
	
\end{lemma}
\begin{remark}
	Lemma \ref{lemma:Delta} provides several possible choices for the design of the set $\mathbb{U}$ based on the knowledge of matrix $Q$ (or homogeneous vectors as shown in Lemma \ref{lemma:Q}). The design of set $\mathbb{U}$ is instrumental for the design of of set $\mathbb{Q}$ and the hybrid sets $\mathcal{F}_o$ and $\mathcal{J}_o$. Note that Lemma \ref{lemma:Delta} a simpler design scheme for the gap $\delta$, with relaxed conditions on the matrix $Q$ as compared to \cite{berkane2017construction} (matrix $A$ is considered in Proposition 2 \cite{berkane2017construction}).
\end{remark}

\begin{lemma} \label{lemma: tr_and_psi}
	Let $ \mathbb{A} = \sum_{i=1}^{n} k_i r_i r_i\T$ with $k_i>0$ and $r_i \in \mathbb{R}^4, i=1,\cdots,n$. Then, for all $g,\bar{g}\in SE(3)$,  the following identities hold:
	\begin{align}
	&\tr  ((I_4-g)\mathbb{A}(I_4-g)\T  ) =     \sum_{i=1}^n k_i \|r_i-g^{-1}r_i\|^2    \label{eqn: vector_measurements_to_trace},\\ 
	&\psi(\mathbb{P}((I_4 - g^{-1})\mathbb{A})) =   \frac{1}{2}   \sum_{i=1}^{n} k_i (g^{-1}r_i  ) \wedge r_i , \label{eqn:vector_measurements_to_psi} \\ 
	&\Ad{\bar{g}}^*    \sum_{i=1}^{n} k_i (\bar{g}g^{-1} r_i  ) \wedge r_i =\sum_{i=1}^{n} k_i  ( g^{-1}r_i  ) \wedge (\bar{g}^{-1} r_i). \label{eqn:vector_measurements_to_Adpsi}
	\end{align}
\end{lemma}

\begin{lemma}\label{lemma:gradient_on_SE_3}	
	Let Assumption \ref{assump:1} hold and consider the following smooth potential function on $SE(3)$:
	\begin{align}
	\mathcal{U}(g) &  = {     \frac{1}{2}}  \tr((I_4-g)\mathbb{A}(I_4-g)\T), \label{eqn:definitinU}
	\end{align}
	where the matrix $\mathbb{A}$ is defined in (\ref{eqn:definition_A}). For any $g=\mathcal{T}(R,p)\in SE(3)$, one has
	\begin{align}	
	&\nabla_g \mathcal{U}(g) :=  g\mathbb{P}((I_4-g^{-1})\mathbb{A}) ,\label{eqn:gradient-g} \\
	&\Psi_{\mathcal{U}}(g) := \{I_4\} {\textstyle  \bigcup}\left\{  g=\mathcal{T}(R,p): R=\mathcal{R}_{a}(\pi,v) , \right.	 \nonumber  \\
	&~~~~~~~~~~~~~~~~~~~  \left.  p= (I_3 - \mathcal{R}_{a}(\pi,v))bd^{-1} ,  v \in \mathcal{E}(Q)\right\}.  \label{eqn:set_Psi_U}
	\end{align}
\end{lemma}
\begin{lemma}\label{lemma:conditions_exp}		
	Consider the potential function (\ref{eqn:definitinU}) under Assumption \ref{assump:1}. Define the following set:
	\begin{align}
	\mathbb{Q} := \left\{\mathcal{T}(R,p) \in SE(3)|  R = \mathcal{R}_a(\theta^*,u) ,\theta^*\in (0,\pi],  p = (I_3- \mathcal{R}_a(\theta^*,u))bd^{-1} , u\in \mathbb{U} \right\} . \label{eqn:definition_set_Q}
	\end{align}
	where $\mathbb{U}= \{u_q | u_q \in \mathbb{S}^2,q=1,\cdots, m\}$.
	There exist strictly positive scalars $\alpha_1, \alpha_2,\alpha_3$ and $\alpha_4$, such that the following inequalities hold:
	\begin{align}
	& \alpha_1 |g|_I^2 \leq \mathcal{U}(g) \leq \alpha_2
	|g|_I^2, ~~~~~~~~~~~~~~~~~~~~g\in SE(3)\label{eqn:bounded_U} \\
	& \alpha_3 |g|_I^2 \leq \|\psi(g^{-1}\nabla_{g}\mathcal{U}(g))\|^2 \leq \alpha_4 |g|_I^2, ~~~g\in \Upsilon\label{eqn:bounded_psi}  \\
	& \|\Ad{g^{-1}}^* \psi(g^{-1}\nabla_{g}\mathcal{U}(g))\|^2 \leq \alpha_4 |g|_I^2, ~~~~~~~g\in \Upsilon. \label{eqn: bouded_Ad_psi}
	\end{align}
	where $\Upsilon: = \{g\in SE(3): \mathcal{U}(g) - \min_{g_q\in \mathbb{Q}} \mathcal{U}(gg_q) \leq \delta\}$ with $\delta < (1-\cos\theta^*) \Delta^*_Q$, $\mathbb{U}$ and $\Delta_Q^*$ designed as per Lemma \ref{lemma:Delta}.
\end{lemma}

In view of (\ref{eqn:output_measurement}), (\ref{eqn: vector_measurements_to_trace}) and (\ref{eqn:definitinU}), let us introduce the following potential function which can be written in terms of the homogeneous output measurements:
\begin{equation}
\mathcal{U}_1({\tilde{g}}) :=  \frac{1}{2}\tr((I_4-{\tilde{g}}) {\mathbb{A}}(I_4-{\tilde{g}})\T) =  \frac{1}{2}   \sum_{i=1}^n k_i\|{r}_i - \hat{g}b_i\|^2,  \label{eqn:definitinU1}
\end{equation}
where the matrix $\mathbb{A}$ is given in (\ref{eqn:definition_A}). Making use of (\ref{eqn:vector_measurements_to_psi}), (\ref{eqn:vector_measurements_to_Adpsi}) in Lemma \ref{lemma: tr_and_psi} and (\ref{eqn:gradient-g}) in Lemma \ref{lemma:gradient_on_SE_3}, one has the following identities:
\begin{align}
\psi ({\tilde{g}}^{-1}\nabla_{{\tilde{g}}} \mathcal{U}_1({\tilde{g}}) )  &= \frac{1}{2}    \textstyle \sum_{i=1}^{n} k_i ({\hat{g}}b_i  ) \wedge r_i, \label{eqn:psi_tilde_g} \\
\Ad{\hat{g}}^* \psi ({\tilde{g}}^{-1}\nabla_{{\tilde{g}}} \mathcal{U}_1({\tilde{g}}) )  &= \frac{1}{2}    \textstyle \sum_{i=1}^{n} k_i  b_i  \wedge ({\hat{g}}^{-1}r_i)
\end{align}

\begin{proposition}\label{pro:1}
	Consider the following hybrid state observer:
	\begin{align}
	&\underbrace{\begin{array}{ll}
		\dot{\hat{g}}~~  = \hat{g} (\xi_y - \hat{b}_a+ k_{\beta}\beta )^\wedge  	\\
		\dot{\hat{b}}_a~  =   -\Gamma \sigma_b
		\end{array}} _{(\hat{g},\hat{b}_a)~\in~ \mathcal{F}_o }
	\underbrace{\begin{array}{ll}
		\hat{g}^{+} ~ = g_q^{-1} \hat{g},  g_q \in \gamma({\hat{g}})      \\[0.1cm]
		\hat{b}_a^{+} ~ = \hat{b}_a
		\end{array}  } _{(\hat{g},\hat{b}_a)~\in~ \mathcal{J}_o } ,  \label{eqn:observer_design1} \\
	&  \beta
	= \frac{1}{2} \Ad{\hat{g}^{-1}}  \textstyle \sum_{i=1}^{n} k_i (\hat{g}b_i  ) \wedge (r_i),   \label{eqn:observer_design_beta1} \\
	&  \sigma_b
	= \frac{1}{2}    \textstyle \sum_{i=1}^{n} k_i  b_i  \wedge ({\hat{g}}^{-1}r_i). \label{eqn:observer_design_sigma1}
	\end{align}  Choose the set $\mathbb{Q}$ designed as per Lemma \ref{lemma:conditions_exp}. Let Assumption \ref{assump:1} and Assumption \ref{assump:2} hold. Then, the results of Theorem \ref{theo:Theorem_1} hold.
\end{proposition}
\begin{proof}
	See Appendix \ref{sec:pro1}
\end{proof}

\begin{remark}
	In view of (\ref{eqn:kenimatic_g}), (\ref{eqn:observer_design}), (\ref{eqn:observer_design_beta1}) and (\ref{eqn:observer_design_sigma1}), the rotational and translational error dynamics in the flow $\mathcal{F}_c$ are given by
	\begin{align}
	\dot{\tilde{R}}~ & = \tilde{R}(-k_\beta (\psi_a(Q\tilde{R})+ {\textstyle \frac{1}{2}} b^\times \tilde{R}\T \tilde{p}_e)  +  (\hat{R}\tilde{b}_{\omega}))^{\times} ,\label{eqn:closed_tilde_R} \\
	\dot{\tilde{p}}~ &= -{     \frac{1}{2}} k_\beta  d\tilde{p}_e + \tilde{R}(\hat{p}^{\times}\hat{R} \tilde{b}_{\omega} + \hat{R}\tilde{b}_v), \label{eqn:closed_tilde_p}
	\end{align}
	where $\tilde{p}_e :=\tilde{p} - (I_3 - \tilde{R} ) bd^{-1}$.
	The error dynamics (\ref{eqn:closed_tilde_R})-(\ref{eqn:closed_tilde_p}) have the same form as Eq. (23) in \cite{hua2011observer}, in the velocity-bias-free case. Note that the dynamics of $\tilde{R}$ and $\tilde{p}$ are coupled as long as $b = dp_c^\mathcal{I}\neq 0$. Therefore, it is expected that noisy or erroneous position measurements would affect the attitude estimation. This motivated us to re-design the estimation scheme in a way that leads to a decoupled rotational error dynamics from the translational error dynamics.
\end{remark}

\section{Decoupling the Rotational Error Dynamics from the Translational Error Dynamics} \label{sec: fully_decoupled}

Define an auxiliary configuration $g_c: = \mathcal{T}(I_3, p_c^\mathcal{I})$ with $p_c^\mathcal{I} =\alpha_i p^{\mathcal{I}}_i$ and $\alpha_i := k_i/\sum_{i=1}^{n_1} k_i $. Consider the modified inertial elements of the homogeneous space $\bar{r}_i$, defined as
$
\bar{r}_i := h(g_c,r_i) =g_c^{-1}r_i, i=1,\cdots,n.
$
Define the modified inertial landmarks as $\bar{p}_i^\mathcal{I} :=  p_i^\mathcal{I} - p_c$.  It is clear that $\sum_{i=1}^{n_1} \alpha_i \bar{p}_i^\mathcal{I} = 0$, which implies that the centroid of the weighted modified landmarks coincides with the origin (see Fig. \ref{fig:digram}). This property is instrumental in achieving decoupled rotational error dynamics from the translational error dynamics. Note that in \cite{vasconcelos2010nonlinear} this property has been achieved through the choice of the parameters $\alpha_i$ assuming that the landmark points are linearly dependent. Our approach does not put such restrictions on the landmarks and the parameters $\alpha_i$.

\begin{figure}[thpb]
	\centering
	\begin{tikzpicture}[xscale=0.56, yscale=0.42]
	\large
	\draw[rotate around={-25:(10 -0.8*1.3,8+0.5*1.3) },fill=black!10] (10 -0.8*1.3,8+0.5*1.3)  ellipse (0.45*1.6 and 0.4*1.6);
	\draw[rotate around={35:(10 -0.8*1.0,8- 0.7*1.0) },fill=black!10] (10 -0.8*1.0,8- 0.7*1.0)  ellipse (0.45*1.2 and 0.4*1.2);
	\draw[rotate around={-25:(10+0.8*1.3,8- 0.5*1.3)},fill=black!10] (10+0.8*1.3,8- 0.5*1.3)  ellipse (0.45*1.3 and 0.4*1.3);
	\draw[rotate around={35:(10 +0.8*1.0,8+ 0.7*1.0) },fill=black!10] (10+0.8*1.0,8+0.7*1.0)  ellipse (0.45*1.6 and 0.4*1.6);
	\draw [line width=3] (10 -0.8*1.3,8+0.5*1.3)  -- (10+0.8*1.3,8- 0.5*1.3);
	\draw [line width=3] (10 -0.8*1.0,8-0.7*1.0)  --  (10+0.8*1.0,8+ 0.7*1.0);
	
	\draw[fill] (7,-1)  circle [radius=0.15] ;
	\draw[fill] (11,-2.5)  circle [radius=0.15] ;
	\draw[fill] (12,0.5) circle [radius=0.15] ;	
	
	\draw [arrow, gray, line width=2.0] (0,0) -- (7*0.98,-1*0.99);	
	\draw [arrow, gray, line width=2.0] (0,0) -- (11*0.98,-2.5*0.99);	
	\draw [arrow, gray, line width=2.0] (0,0) -- (12*0.98,0.5*0.99);
	\node at (6,-0.3)  {$p_1^{\mathcal{I}}$};
	\node at (5,-1.8) {$p_2^{\mathcal{I}}$};
	\node at (5 ,0.8) {$p_3^{\mathcal{I}}$};
	
	\draw [arrow, cyan, line width=2.0] (10,8) -- (7 ,-1*0.88) node at (7 + 2*0.5,-1 +9*0.5)  {$p_1^{\mathcal{B}}$};
	\draw [arrow,  cyan, line width=2.0] (10,8) -- (11,-2.5*0.95) node at (9.8+ 0.5*0.5,-2.5 +10.5*0.5) {$p_2^{\mathcal{B}}$};
	\draw [arrow,  cyan, line width=2.0] (10,8) -- (12 ,0.5*0.98) node at (12.7 -2*0.5,-0.5 +8.5*0.5) {$p_3^{\mathcal{B}}$};	
	
	\draw[fill,black!50] (9.8,-0.5)  circle [radius=0.15] node at (9.3,-0.8) {$p_c^{\mathcal{I}}$};
	\draw [arrow,dashed,black!20,line width=2.0] (9.8,-0.5) -- (7*1.01,-1*0.99);
	\node at (8.0 + 1*0.1,-1 +0.4)  {$\bar{p}_1^{\mathcal{I}}$};
	\draw [arrow,dashed, black!20, line width=2.0] (9.8,-0.5) -- (11 ,-2.5*0.95);
	\node at (9.8 + 1*0.5,-2.0 +1*0.8) {$\bar{p}_2^{\mathcal{I}}$};
	\draw [arrow,dashed, black!20,line width=2.0] (9.8,-0.5) -- (12*0.99,0.5*0.99);
	\node at (11.5-1*0.4,1 -1*0.9) {$\bar{p}_3^{\mathcal{I}}$};
	
	\draw[fill] (0,0) circle [radius=0.15] node at (-0,0.8) {\large $\{\mathcal{I}\}$} ;
	\draw [arrow, thick] (0,0) -- (1*2.5,0*2.5) node at (1*2.8,0*2.8)  {$e_1$};
	\draw [arrow, thick] (0,0) -- (0*2.4,-1*2.4) node at (0*2.5,-1*2.5) {$e_3$};
	\draw [arrow, thick] (0,0) -- (-0.8*2.5,-0.6*2.5) node at (-0.8*2.8,-0.6*2.8)  {$e_2$};
	
	\draw[fill] (10,8) circle [radius=0.15] node at (9.9 ,8+1.5) {\large $\{\mathcal{B}\}$};
	\draw [arrow, thick] (10,8) -- (10-0.8*2.5,8+0.5*2.5) node at (10-0.8*2.8,8+0.5*2.8) {$e_2$};
	\draw [arrow, thick] (10,8) -- (10-0.8*2.5,8- 0.7*2.5) node at (10-0.8*2.8,8- 0.7*2.8) {$e_1$};
	\draw [arrow, thick] (10,8) -- (10+0  *2.5,8-   1*2.5)  node at (10+0*2.8, 8-1*2.8) {$e_3$};	
	
	\draw[fill=black!50] (9.8,-0.5)  circle [radius=0.15] node at (9.8,0.5) { $\{\mathcal{I}'\}$} ;
	\draw [arrow, thick] (9.8,-0.5) -- (9.8+1*2.5,-0.5+0*2.5) node at (9.8+1*2.8,-0.5+0*2.8)  {$e_1'$};
	\draw [arrow, thick] (9.8,-0.5) -- (9.8+0*2.4,-0.5-1*2.4) node at (9.8+0*2.5,-0.5-1*2.5) {$e_3'$};
	\draw [arrow, thick] (9.8,-0.5)  -- (9.8+-0.8*2.5,-0.5-0.6*2.5) node at (9.8-0.8*2.8,-0.5-0.6*2.8)  {$e_2'$};
	\end{tikzpicture}
	\caption{The landmarks coordinates in the inertial frame and body frame are represented with Solid lines. The landmarks coordinates in the auxiliary frame are represented with dotted lines.}
	\label{fig:digram}
\end{figure}
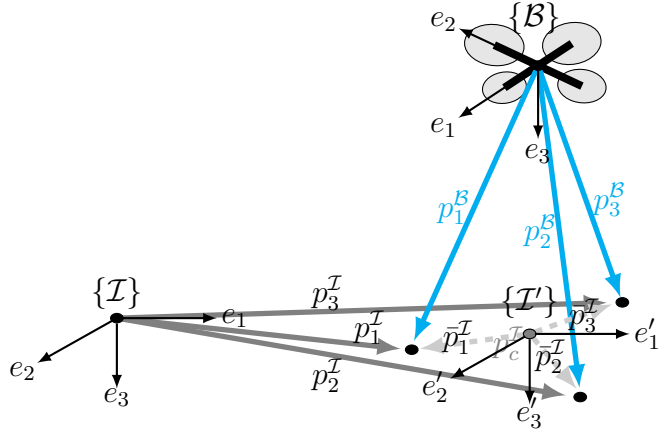
	
Define the modified pose and pose estimate as $\underline{g} :=\mathcal{T}(\underline{ {R}},\underline{ {p}})= {g}_c^{-1}g$ and $\underline{\hat{g}} := \mathcal{T}(\underline{\hat{R}},\underline{\hat{p}})={g}_c^{-1}\hat{g}$. One verifies that $b_i  = h(g_c^{-1}g,\bar{r}_i) = \underline{g}^{-1}\bar{r}_i$. Define the new pose estimation error $\underline{\tilde{g}}=\mathcal{T}(\underline{\tilde{R}},\underline{\tilde{p}}): =\underline{g} \underline{\hat{g}}^{-1} = {g}_c^{-1}g\hat{g}^{-1}{g}_c$ with $\underline{\tilde{R}}=\tilde{R}$ and $\underline{\tilde{p}}=\tilde{p}-(I-\tilde{R})p_c$. It is clear that that $\tilde{g}$ tends to $I_4$ if $\underline{\tilde{g}}$ tends to $I_4$. Let us introduce the following potential function:
\begin{equation}
\mathcal{U}_2(\underline{\tilde{g}}):=  \frac{1}{2}\tr((I_4-\underline{\tilde{g}}) \bar{\mathbb{A}}(I_4-\underline{\tilde{g}})\T) = \frac{1}{2}   \sum_{i=1}^n k_i\|\bar{r}_i - \underline{\hat{g}}b_i\|^2,  \label{eqn:definitinU2}
\end{equation}
where $\bar{\mathbb{A}}: = \sum_{i=1}k_i \bar{r}_i\bar{r}_i\T=\text{diag}(Q,d)$. In view of (\ref{eqn:property2}), (\ref{eqn:definitinU1}) and (\ref{eqn:definitinU2}), one can show that, for any $\tilde{g}\in SE(3)$
\begin{align*}
\mathcal{U}_2(\underline{\tilde{g}})
&=  \frac{1}{2}\tr((I_4-g_c^{-1}{\tilde{g}}g_c) g_c^{-1}\mathbb{A}g_c^{-\top} (I_4-g_c^{-1}{\tilde{g}}g_c)\T) \\
&=  \frac{1}{2}\tr(g_c^{-1}(I_4- {\tilde{g}})  \mathbb{A}  (I_4- {\tilde{g}} )\T g_c^{-\top} )  \\
&=  \frac{1}{2}\tr((I_4- {\tilde{g}})  \mathbb{A}  (I_4- {\tilde{g}} )\T ) 
\end{align*}
which implies
$\mathcal{U}_2(\underline{\tilde{g}})=  \mathcal{U}_1(\tilde{g})$. In the sequel, we will make use of $\mathcal{U}_2(\underline{\tilde{g}})$ and $\mathcal{U}_1(\tilde{g})$ equivalently. Making use of (\ref{eqn:vector_measurements_to_psi}), (\ref{eqn:vector_measurements_to_Adpsi})  in Lemma \ref{lemma: tr_and_psi} and (\ref{eqn:gradient-g}) in Lemma \ref{lemma:gradient_on_SE_3}, one can also show that
\begin{align}
\psi (\underline{\tilde{g}}^{-1}\nabla_{\underline{\tilde{g}}} \mathcal{U}_{2}(\underline{\tilde{g}}) )  &=   \frac{1}{2}    \sum_{i=1}^{n} k_i (g_c^{-1}\hat{g}b_i  ) \wedge (g_c^{-1}r_i),\label{eqn:psi_underline_tilde_g} \\
\Ad{g_c^{-1}\hat{g}}^*   \psi (\underline{\tilde{g}}^{-1}\nabla_{\underline{\tilde{g}}} \mathcal{U}_{2}(\underline{\tilde{g}}) )  &=   \frac{1}{2}    \sum_{i=1}^{n} k_i b_i  \wedge (\hat{g}^{-1}r_i).
\end{align}

Define the extended state $x':=(\underline{\tilde{g}},\tilde{b},\hat{g},\hat{b}_a,t)\in \mathcal{S}$ and the closed set $\bar{\mathcal{A}}'  := \{x'\in   \mathcal{S}: \underline{\tilde{g}}=I_4, \tilde{b}_a=0\}$. Let $|x'|_{\bar{\mathcal{A}}'}$ denote the distance to the set $\bar{\mathcal{A}}'$ such that
$
|x'|_{\bar{\mathcal{A}}'}^2 := \inf_{y=(I_4,0,\bar{g},\bar{b}_a,\bar{t}) \in \bar{\mathcal{A}}'} (\|I_4-\underline{\tilde{g}}\|_F^2 + \|\tilde{b}_a\|^2 + \|\bar{g}-\hat{g}\|_F^2 + \|\bar{b}_a-\hat{b}_a\|^2+ \|\bar{t}-t\|^2)=|\underline{\tilde{g}}|_I^2 + \|\tilde{b}_a\|^2$.

\begin{proposition}\label{pro:2}
	Consider the following hybrid state observer:
	\begin{align}
	&\underbrace{\begin{array}{l}
		\dot{\hat{g}} ~ = \hat{g} (\xi_y - \hat{b}_a+ k_{\beta}\beta )^\wedge  	\\
		\dot{\hat{b}}_a  =  -\Gamma \sigma_b
		\end{array}} _{({\hat{g}},\hat{b}_a)~\in~ \mathcal{F}_o}
	\underbrace{\begin{array}{l}
		\hat{g}^{+}  = g_q^{-1} \hat{g},  g_q \in \gamma({\hat{g}})      \\[0.1cm]
		\hat{b}_a^{+}  = \hat{b}_a
		\end{array}  } _{({\hat{g}},\hat{b}_a)~\in~ \mathcal{J}_o } ,  \label{eqn:observer_design2} \\
	& ~~\beta~=   \frac{1}{2} \Ad{\hat{g}^{-1}g_c}     \sum_{i=1}^{n} k_i (g_c^{-1}\hat{g}b_i  ) \wedge (g_c^{-1}r_i)  \label{eqn:observer_design_beta2} \\
	& ~~\sigma_b=   \frac{1}{2}    \sum_{i=1}^{n} k_i b_i  \wedge (\hat{g}^{-1}r_i). \label{eqn:observer_design_sigma2}
	\end{align}
	Choose the set $\mathbb{Q}$ designed as per Lemma \ref{lemma:conditions_exp}. Let Assumption \ref{assump:1} and Assumption \ref{assump:2} hold. Then, the number of jumps is finite and for any initial condition $x'(0,0)\in \mathcal{S}$, the solution $x'(t,j)$ is complete and there exist $\bar{k}>0$ and $\bar{\lambda}>0$ such that
	\begin{equation}
	|x'(t,j)|_{\bar{\mathcal{A}}'}\leq \bar{k} \exp(-\bar{\lambda} (t+j)) |x'(0,0)|_{\bar{\mathcal{A}}'},
	\end{equation}
for all $(t,j)\in \dom{x'}$.
\end{proposition}
\begin{proof}
	The proof is similar as Proposition \ref{pro:1}.
\end{proof}

\begin{remark}
		Interestingly, in view of (\ref{eqn:vector_measurements_to_psi})-(\ref{eqn:vector_measurements_to_Adpsi}),  one obtains the following expressions of $\beta$ as follows:
		\begin{align}
		&\frac{1}{2} \Ad{\hat{g}^{-1}g_c}    \sum_{i=1}^{n} k_i (g_c^{-1}\hat{g}b_i  ) \wedge (g_c^{-1}r_i)  = \frac{1}{2} \Ad{\hat{g}^{-1}}\Ad{g_c}  \Ad{g_c}^* \sum_{i=1}^{n} k_i ( \hat{g}b_i  ) \wedge ( r_i)    \label{eqn:new_beta}
		\end{align}
	It is worth pointing that the sole difference between this new observer and the observer in Proposition 1 lies in the definition terms $\beta$ (see (\ref{eqn:observer_design_beta1}) and \ref{eqn:new_beta}).
	In view of (\ref{eqn:kenimatic_g}), (\ref{eqn:observer_design2})-(\ref{eqn:observer_design_sigma2}), the error dynamics can be written as
	\begin{align}
	\dot{\underline{\tilde{R}}}~ & = \underline{\tilde{R}}(-k_\beta\psi_a(Q\underline{\tilde{R}}) +  (\underline{\hat{R}}\tilde{b}_{\omega}))^{\times},  \label{eqn:dynamic_tilde_R}\\
	\dot{\underline{\tilde{p}}} ~&=  - {  \frac{1}{2}}k_\beta d \underline{\tilde{p}}+ \underline{\tilde{R}}(  \underline{\hat{p}}^{\times}\underline{\hat{R}} \tilde{b}_{\omega} + \underline{\hat{R}}\tilde{b}_v), \label{eqn:dynamic_tilde_p} \\
	\dot{\tilde{b}}_\omega & =   -k_\omega \left(\underline{\hat{R}}\T \psi_a(Q\underline{\tilde{R}})   -{  \frac{1}{2}}d \underline{\hat{R}}\T \underline{\hat{p}}^\times  \underline{\tilde{R}}\T\underline{\tilde{p}}\right)  \label{eqn:dot_tilde_b_omega}\\
	\dot{\tilde{b}}_v & =  	-{  \frac{1}{2}}k_v d\underline{\hat{R}}\T\underline{\tilde{R}}\T \underline{\tilde{p}}  \label{eqn:dynamic_tilde_b_a}.
	\end{align}
	It is clear that, in the velocity bias-free case, in contrast to \eqref{eqn:closed_tilde_R}, the dynamics of $\underline{\tilde{R}}$ does not depend on $\underline{\tilde{p}}$ as shown in (\ref{eqn:dynamic_tilde_R}), and $\underline{\tilde{p}}$ enjoys exponential stability when $\tilde{b}_a = 0$ as it can be seen from (\ref{eqn:dynamic_tilde_p}). However, when the velocity bias is considered, the rotational error dynamics is affected by the estimated position $\hat{p}$ involved in the dynamics of $\tilde{b}_\omega$ in (\ref{eqn:dot_tilde_b_omega} ).
	In order to achieve the decoupled property, in the case where velocity bias is not neglected, the following modified estimation scheme is proposed.
	
\end{remark}

Let us consider the following modified estimation scheme:
\begin{align}
&\underbrace{\begin{array}{ll}
	\dot{\hat{g}}   ~= \hat{g} (\xi_y - \hat{b}_a+ k_{\beta}\beta )^\wedge \\
	\dot{\hat{b}}_a = -\Gamma \sigma_b
	\end{array}}_{({\hat{g}},\hat{b}_a)~\in~ \mathcal{F}_o}
\underbrace{\begin{array}{ll}
	\hat{g}^{+}   ~= g_q^{-1} \hat{g},\    g_q \in \gamma(\hat{g})      \\[0.1cm]
	\hat{b}_a^{+}  ~ = \hat{b}_a
	\end{array} }_{({\hat{g}},\hat{b}_a~\in~ \mathcal{J}_o} ,  \label{eqn:observer_design3} \\
& ~~\beta~= \frac{1}{2}\Ad{\hat{g}^{-1}g_c}  \textstyle \sum_{i=1}^{n} k_i (g_c^{-1}\hat{g}b_i  ) \wedge (g_c^{-1}r_i), \label{eqn:observer_design_beta3}\\
&   ~~\sigma_b=\frac{1}{2} \Lambda\T  \textstyle \sum_{i=1}^{n} k_i (g_c^{-1}\hat{g}b_i  ) \wedge (g_c^{-1}r_i).
\label{eqn:observer_design_sigma3}
\end{align}
where $\Lambda :=\text{diag}(\underline{\hat{R}}, \underline{\hat{R}})$, $ \hat{g}(0) \in SE(3), \hat{b}_\omega(0),\hat{b}_v(0)\in \mathbb{R}^3$,  and $k_\omega,k_v, k_{\beta}$ and $k_b$ are strictly positive scalars.  In view of (\ref{eqn:kenimatic_g}), (\ref{eqn:observer_design3})-(\ref{eqn:observer_design_sigma3}) one can write the closed-loop system as an autonomous hybrid system.
\begin{align}
&\mathcal{H}': \begin{cases}
\dot{x}'~~\in F(x') & x'\in \mathcal{F}_c':=\{x'\in \mathcal{S}: (\hat{g},\hat{b}_a)\in \mathcal{F}_o\}  \\
x'^+\in G(x') & x'\in \mathcal{J}_c':=\{x'\in \mathcal{S}: (\hat{g},\hat{b}_a)\in \mathcal{J}_o\}
\end{cases}   \label{eqn:closed_loop4}
\end{align}
with
\begin{align}
	F(x')= \begin{bmatrix}
	\underline{\tilde{g}}   (\Ad{\underline{\hat{g}}}\tilde{b}_a - k_\beta  \psi(\underline{\tilde{g}}^{-1}\nabla_{\underline{\tilde{g}}}\mathcal{U}_2(\underline{\tilde{g}}))   )^\wedge \\
	-\Gamma \Lambda\T \psi(\underline{\tilde{g}}^{-1}\nabla_{\underline{\tilde{g}}} \mathcal{U}_2(\underline{\tilde{g}}) )\\
	\underline{\hat{g}} (\xi_y - \hat{b}_a+k_\beta \Ad{\underline{\hat{g}}^{-1}}  \psi(\underline{\tilde{g}}^{-1}\nabla_{\underline{\tilde{g}}} \mathcal{U}_2(\underline{\tilde{g}}) )  )^\wedge   \\
	-\Gamma \Lambda\T \psi(\underline{\tilde{g}}^{-1}\nabla_{\underline{\tilde{g}}} \mathcal{U}_2(\underline{\tilde{g}}) )\\
	1
	\end{bmatrix}, G(x') = \begin{bmatrix}
	\underline{\tilde{g}}\bar{g}_q \\[0.05cm]
	\tilde{b}_a \\
	g_q^{-1}\hat{g} \\
	\hat{b}_a \\
	t
	\end{bmatrix}.   \nonumber
\end{align}
where $\bar{g}_q:=g_c^{-1}g_q g_c$. Note that the closed-loop system (\ref{eqn:closed_loop4}) also satisfies the hybrid basic conditions of \cite{goebel2009hybrid}. Now, one can state the following theorem:
\begin{theorem}\label{theo:Theorem_2}
	Consider the closed-loop system (\ref{eqn:closed_loop4}) with potential function $\mathcal{U}_2$ in (\ref{eqn:definitinU2}). Choose the set $\mathbb{Q}$ designed as per Lemma \ref{lemma:conditions_exp}. Let Assumption \ref{assump:1} and Assumption \ref{assump:2} hold. Then, the number of jumps is finite and for any initial condition $x'(0,0)\in \mathcal{S}$, the solution $x'(t,j)$ is complete and there exist $k'>0$ and $\lambda'>0$ such that
	\begin{equation}
	|x'(t,j)|_{\bar{\mathcal{A}}'}\leq k' \exp(-\lambda' (t+j)) |x'(0,0)|_{\bar{\mathcal{A}}'},
	\end{equation}
	for all $(t,j)\in \dom{x'}$.
\end{theorem}
\begin{proof}
	See Appendix \ref{sec: Theorem2}.
\end{proof}
\begin{remark}
	From (\ref{eqn:vector_measurements_to_Adpsi}), (\ref{eqn:observer_design_beta3}) and (\ref{eqn:observer_design_sigma3}), one can show that
	\begin{align}
	& \beta =  \frac{1}{2}  \Ad{\hat{g}^{-1}}\Ad{g_c}  \Ad{g_c}^*   \textstyle \sum_{i=1}^{n} k_i ( \hat{g}b_i  ) \wedge ( r_i)    \label{eqn:new_beta2}\\
	&\sigma_b =  \frac{1}{2} \Lambda\T \Ad{g_c} ^*  \textstyle  \sum_{i=1}^{n} k_i (\hat{g}b_i  ) \wedge (r_i)  . \label{eqn:new_sigma_b}
	\end{align}
	In view of (\ref{eqn:observer_design1})-(\ref{eqn:observer_design_sigma1}) and (\ref{eqn:observer_design3}), (\ref{eqn:new_beta2})-(\ref{eqn:new_sigma_b}), one can notice that the difference between the observer in Theorem \ref{theo:Theorem_2} and the observer in Proposition \ref{pro:1} is related to the terms $\beta$ and $\sigma_b$. It is worth pointing out that the observer in Theorem \ref{theo:Theorem_2}, without ``hybridation'' (\textit{i.e.,} in the flow set), is not gradient-based as in \cite{lageman2010gradient,hua2011observer,hua2015gradient,khosravian2015observers}.
\end{remark}
\begin{remark}
	In view of  (\ref{eqn:vector_measurements_to_Adpsi}), (\ref{eqn:gradient-g}) and (\ref{eqn:closed_loop4}), one has the following error dynamics in the flows:
	\begin{align}
	\dot{\underline{\tilde{R}}} ~& = \underline{\tilde{R}}  (-k_\beta\psi_a(Q\underline{\tilde{R}}) + \underline{\hat{R}}\tilde{b}_{\omega})^{\times} ,  \label{eqn:closed_tilde_R3}\\
	\dot{\underline{\tilde{p}}} ~&=  -{     \frac{1}{2}}k_\beta d\underline{\tilde{p}}+ \underline{\tilde{R}}(  \underline{\hat{p}}^{\times}\underline{\hat{R}} \tilde{b}_{\omega} + \underline{\hat{R}}\tilde{b}_v), \label{eqn:closed_tilde_p3} \\
	\dot{\tilde{b}}_\omega & =  -k_\omega   \underline{\hat{R}}\T \psi_a(Q\underline{\tilde{R}})  , \label{eqn:dynamic_tilde_b_omega3} \\
	\dot{\tilde{b}}_v & =  - {      \frac{1}{2}}k_v  d \underline{\hat{R}}\T  \underline{\tilde{R}} \T \underline{\tilde{p}}  . \label{eqn:dynamic_tilde_b_v3}
	\end{align}
	Using the facts $\underline{\tilde{R}}=\tilde{R}, \underline{\tilde{p}}=\tilde{p}-(I_3-\tilde{R})p_c^\mathcal{I}, \underline{\hat{R}}=\hat{R}$ and $\underline{\hat{p}}=\hat{p}-p_c^\mathcal{I}$, one can notice that the rotational error dynamics (\ref{eqn:closed_tilde_R3}) together with (\ref{eqn:dynamic_tilde_b_omega3}) do not depend on the translational estimation, which guarantees the aimed at decoupling property. 
\end{remark}
\section{Simulation} \label{sec:simulation}
In this section, some simulation results are presented to illustrate the performance of the proposed hybrid pose observers (observer in Proposition \ref{pro:1}, observer in Proposition \ref{pro:2} and  observer in Theorem \ref{theo:Theorem_2} and referred to,respectively, as H, HD1 and HD2). We refer to the smooth  non-hybrid observer (\textit{i.e.,} the observer in Proposition \ref{pro:1} without the switching mechanism) as S.

As commonly used in practical applications, to avoid the bias estimation drift, in the presence of measurement noise, we introduce the following projection mechanism \cite{ioannou1995robust}:
\begin{align*}
\proj_{\Delta}(\hat{b}, \Gamma \sigma)  := \begin{cases}
\Gamma \sigma, &\text{if } \hat{b} \in \Pi_\Delta \text{ or }  \nabla_{\hat{b}} \mathcal{P}\T \Gamma \sigma \leq 0  \\
\left(I - \varrho(\hat{b})\Gamma \frac{\nabla_{\hat{b}} \mathcal{P}\nabla_{\hat{b}} \mathcal{P}\T}{\nabla_{\hat{b}} \mathcal{P}\T\Gamma \nabla_{\hat{b}} \mathcal{P}}\right) \Gamma \sigma,  &  \text{otherwise}
\end{cases},
\end{align*}
where $\hat{b},\sigma \in \mathbb{R}^n, \Gamma \in \mathbb{R}^{n\times n}$, $\mathcal{P}(\hat{b}) := \|\hat{b}\| - \Delta$, $\Pi_\Delta = \{\hat{b}| \mathcal{P}(\hat{b}) \leq 0\}$, $\Pi_{\Delta,\epsilon} = \{\hat{b} | \mathcal{P}(\hat{b}) \leq \epsilon\}$ and $\varrho(\hat{b}) := \min\{1,  \mathcal{P}(\hat{b}) /\epsilon\}$ for some positive parameters $\Delta$ and $\epsilon$. Given $\|\hat{b}(0)\| < \Delta$,
one can verify that the projection map $\proj_{\Delta}$ satisfies the following properties:
\begin{itemize}
	\item [1)] $\|\hat{b}(t)\|\leq \Delta + \epsilon $, for all $t \geq 0$  ;
	\vspace*{0.1cm}
	\item [2)] $\tilde{b}\T \Gamma^{-1} \proj_{\Delta}(\hat{b},\Gamma \sigma )\leq \tilde{b}\T  \sigma$;
	\vspace*{0.1cm}
	\item[3)] $ \|\proj_{\Delta}(\hat{b},\Gamma \sigma )\|  \leq \|\Gamma \sigma\| $.
\end{itemize}

Consider the three inertial vectors
$v_1^\mathcal{I}=[0~0~1]\T, v_2^\mathcal{I}=[\frac{\sqrt{3}}{2}~\frac{1}{2}~0]\T, v_3^\mathcal{I} = [-\frac{1}{2}~\frac{\sqrt{3}}{2}~0]\T$ and one landmark  $p^\mathcal{I}=[\frac{\sqrt{2}}{2}~ \frac{\sqrt{2}}{2}~ 2]\T$ are available. The initial pose for all the observers is taken as the identity \ie,  $\hat{g}(0)=I_4$. The system's initial conditions are taken as follows: $R(0) = \mathcal{R}_a(\pi,v)$ with $v=[1~0~0]\T \in \mathcal{E}(Q)$ and $p(0)=[0~1~ 4]\T$. The system is driven by the following linear and angular velocities: $v(t)= 2[\cos(t) ~ \sin(t)~0]\T$, $\omega(t)=[-\sin(t) ~ \cos(t) ~ 0]\T$. For the hybrid design, we choose $\theta^* = 2\pi/3$, $\delta = 1$ and $\mathbb{U}=\mathbb{E}(Q)$. The gain parameters involved in all the observers are taken as follows: $k_i= 1, i=1,\cdots,4$, $k_{\beta} = 1, k_\omega=1, k_v = 1$.

 The simulation results are given in Fig. \ref{fig:simulation1} - Fig. \ref{fig:simulation2}, from which one can clearly see the improved performance of the decoupled hybrid observer as compared to the non-decoupled hybrid observer and non-hybrid observer.
 
 \begin{figure}[H]
 	\centering
 	\subfloat[Rotation and position estimation errors.]{\includegraphics[width=0.41\linewidth]{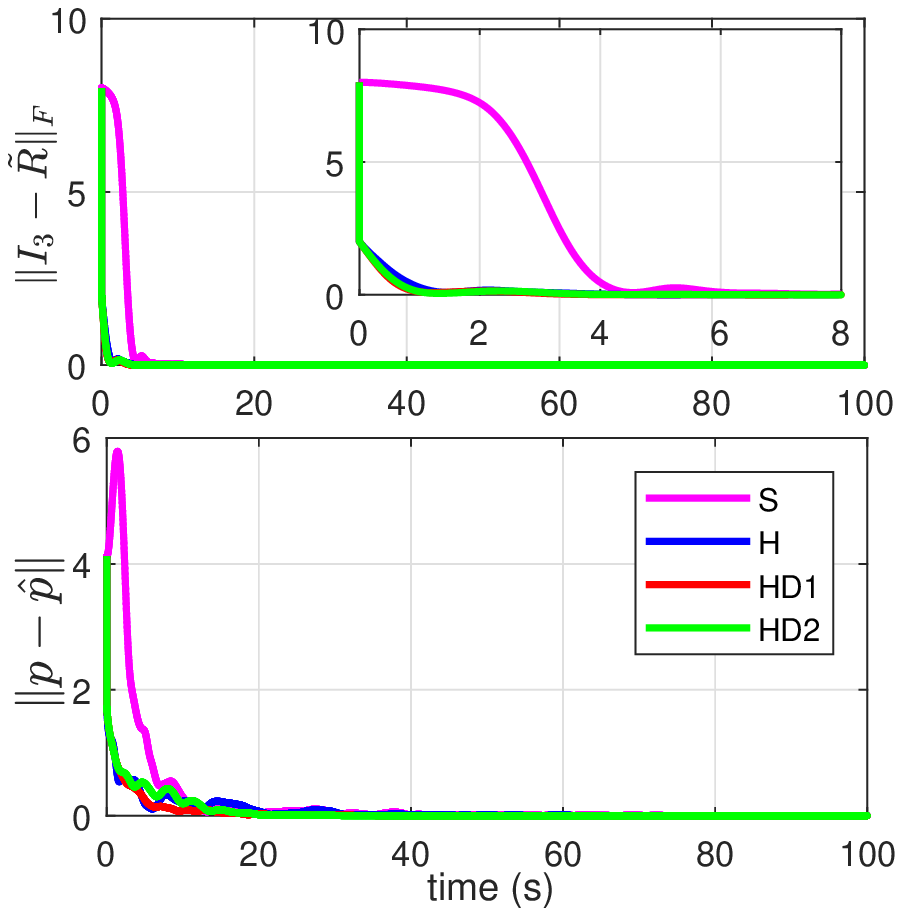}} 
 	\subfloat[Velocity-bias estimation errors.]{\includegraphics[width=0.41\linewidth]{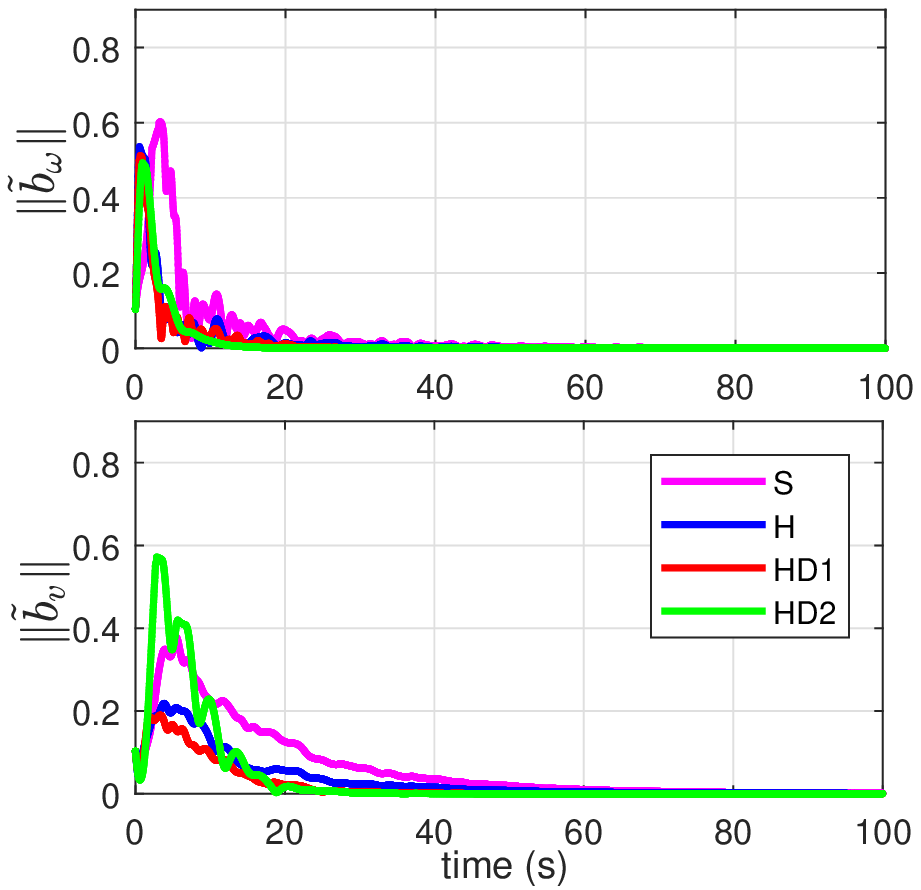}}
 	\caption{Simulation results with non-noisy output measurements and constant velocity-bias $b_{a} = [-0.02 ~ 0.02 ~~ 0.1~~0.2 ~ -0.1 ~ 0.01]\T$.}
 	\label{fig:simulation1}
 \end{figure}
 
 \begin{figure}[H]
 	\centering
 	\subfloat[Rotation and position estimation errors.]{\includegraphics[width=0.41\linewidth]{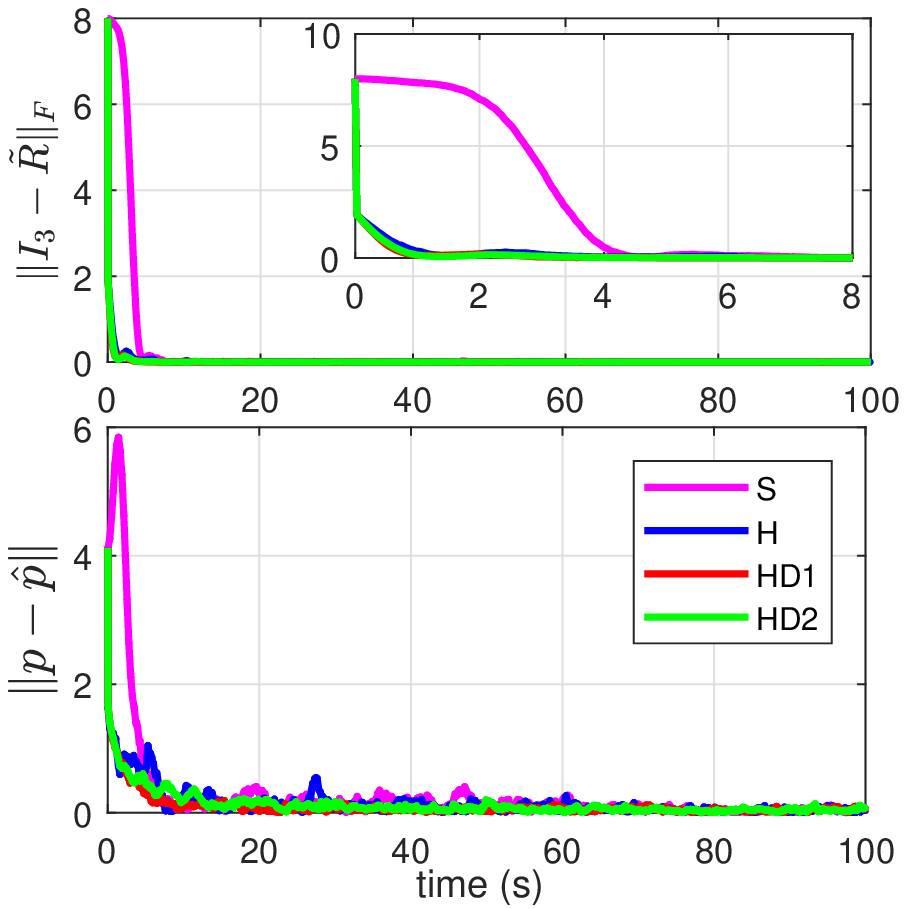}} 
 	\subfloat[Velocity-bias estimation errors.]{\includegraphics[width=0.41\linewidth]{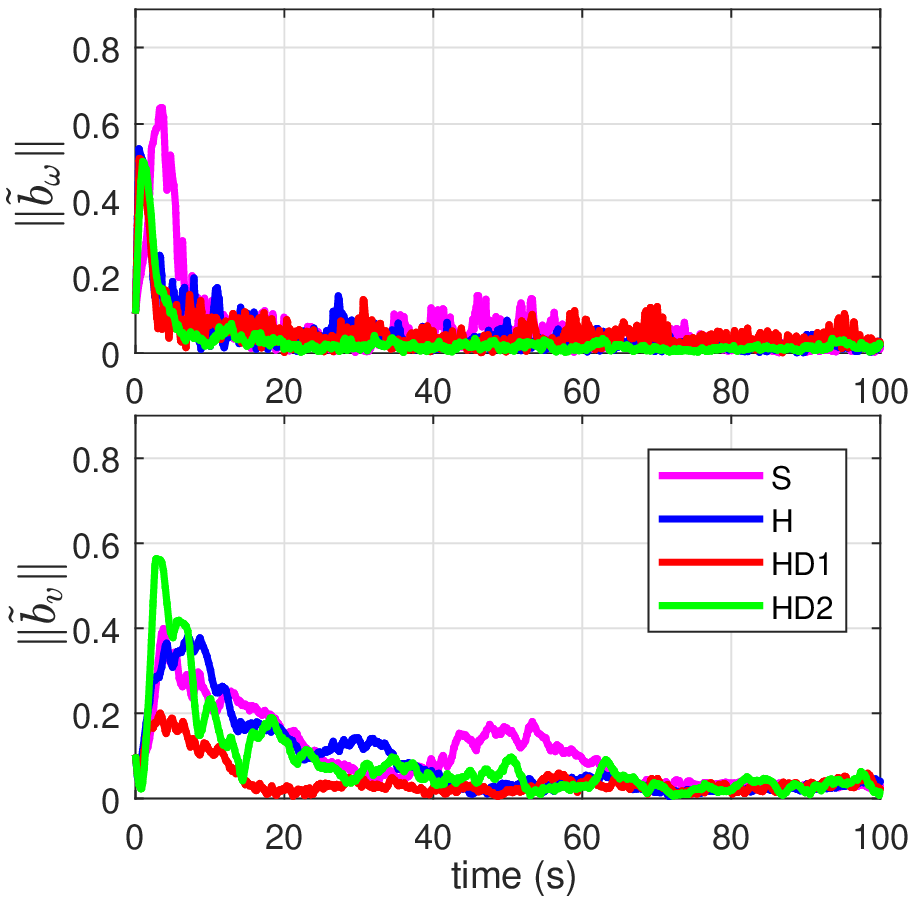}}
 	\caption{Simulation results with additive white Gaussian noise of variance 0.1 in the output measurements and time-varying velocity-bias $b_{a}(t) =\cos(0.02t) [-0.02 ~ 0.02 ~~ 0.1~~0.2 ~ -0.1 ~ 0.01]\T $.}
 	\label{fig:simulation2}
 \end{figure}

\section{Conclusion} \label{sec:conclusion}
A globally exponentially stable hybrid pose and velocity-bias estimation scheme evolving on $SE(3)\times \mathbb{R}^6$ has been proposed. The proposed observer is formulated in terms of homogeneous output measurements of known inertial vectors and landmark points. It relies on an observer-state jump mechanism designed to avoid the undesired critical point while ensuring a decrease of the potential function in the flow and jump sets. Moreover, an auxiliary coordinate transformation is introduced on the landmark measurements, and a modified observer, leading to a decoupled rotational error dynamics from the  translational dynamics, is proposed.

\bibliographystyle{ieeetran}
\bibliography{mybib}

\section{Appendices} \label{sec:appendix}

\subsection{Proof of Theorem \ref{theo:Theorem_1}}\label{sec:theo1}
	Let us consider the following real-valued function:
	\begin{equation}
	V(x) = \mathcal{U}(\tilde{g}) + \frac{1}{2}\tilde{b}_a \Gamma^{-1} \tilde{b}_a,
	\end{equation}
	Taking the time derivative of $V$, along the flows of trajectories of (\ref{eqn:closed_loop}), one can show that
	\begin{align}
	\dot{V}&=\left\langle \nabla_{\tilde{g}} \mathcal{U}(\tilde{g}),  \tilde{g}(\Ad{\hat{g}}(\tilde{b}_a - k_\beta  \beta ) )^\wedge \right\rangle_{\tilde{g}}  - \tilde{b}_a\T \Gamma^{-1}   \Gamma \sigma_b    \nonumber \\
	&\leq  \langle\langle \tilde{g}^{-1}\nabla_{\tilde{g}} \mathcal{U}(\tilde{g}),    (\Ad{\hat{g}}(\tilde{b}_a - k_\beta  \beta(\tilde{g}))  )^\wedge  \rangle\rangle -\tilde{b}_a\T  \sigma_b  \nonumber \\
	&\leq  -k_\beta \langle\langle \tilde{g}^{-1}\nabla_{\tilde{g}} \mathcal{U}(\tilde{g}),   (\psi(\tilde{g}^{-1} \nabla_{\tilde{g}} \mathcal{U}(\tilde{g})))^\wedge  \rangle\rangle  +  \langle\langle \Ad{\hat{g}}^*~ \tilde{g}^{-1}\nabla_{\tilde{g}} \mathcal{U}(\tilde{g}),    (\tilde{b}_a)^\wedge  \rangle\rangle -\tilde{b}_a\T  \sigma_b(\tilde{g}) \nonumber \\
	&\leq -k_\beta  \|\psi(\tilde{g}^{-1} \nabla_{\tilde{g}} \mathcal{U}(\tilde{g})) \|^2   + \tilde{b}_a\T \Ad{\hat{g}}^*  \psi(\tilde{g}^{-1} \nabla_{\tilde{g}} \mathcal{U}(\tilde{g})) - \tilde{b}_a\T  \sigma_b  \nonumber \\
	&=  -k_\beta  \| \psi(\tilde{g}^{-1} \nabla_{\tilde{g}} \mathcal{U}(\tilde{g})) \|^2 , \label{eqn:dot_V}
	\end{align}
	where we made use of (\ref{eqn:psi_property}) and $\Ad{\hat{g}} \Ad{\hat{g}^{-1}} = I_6$. Then, for any $ x\in \mathcal{F}_c$, $V$ is non-increasing  along the flows of (\ref{eqn:closed_loop}). Using the fact $\mathcal{F}_c\cup \mathcal{J}_c=\mathcal{S}$ and the definitions of $\mathcal{X}_{\mathcal{U}}$ and $\mathcal{J}_c$, from (\ref{eq2_theo1}) one has $\mathcal{X}_\mathcal{U} \times \mathbb{R}^6 \times SE(3)\times \mathbb{R}^6 \times \mathbb{R}_{\geq 0}\subseteq \mathcal{J}_c$. 
	
	For any $x\in \mathcal{J}_c$, one can show that
	\begin{align}
	V(x^+)  - V(x) & = \mathcal{U}(\tilde{g}^+) - \mathcal{U}(\tilde{g}) \nonumber \\
	& \leq \min_{g_q\in \mathbb{Q}}  \mathcal{U}(\tilde{g}g_q) - \mathcal{U}(\tilde{g}) \leq - \delta, \label{eqn:decrease_V}
	\end{align}
	which implies that $V$ is strictly decreasing over the jumps of (\ref{eqn:closed_loop}). Thus, in view of (\ref{eqn:dot_V}) and (\ref{eqn:decrease_V}), one can easily show that
	\begin{equation}
	0 < V(x(t,j)) \leq V(x(0,0)) - j \delta ,
	\end{equation}
	which leads to $0\leq j\leq J_{\max}: = \left\lceil{V(x(0,0))}/{\delta}~\right\rceil$ with $\lceil \cdot \rceil$ denoting the ceiling function. This implies that the number of jumps is finite.
	
	To show exponential stability, let us consider the following Lyapunov function candidate:
	\begin{align}
	\mathcal{L}(x)  =  V(x) - z\T U \tilde{b}_a,
	\end{align}
	where $U := \text{diag}(\mu_1I_3, \mu_2 I_3)$ with $ \mu_1 ,\mu_2 > 0$ and $z := [\psi_a(\tilde{R})\T \hat{R}, \tilde{p}\T R ]\T$. Let $e :=[|\tilde{g}|_I, \|\tilde{b}_\omega\|, \|\tilde{b}_v\|]\T$ and $e_i$ be the $i$-th elements of $e$. From \eqref{eq1_theo1}, one obtains
		\begin{align*}
		\mathcal{L}
		& \leq    \alpha_2 e_1^2  + \frac{1}{k_\omega}e_2^2 + \frac{1}{k_v}e_3^2 + e_1(\frac{\sqrt{2}\mu_1 }{2} e_2+ \mu_2 e_3) \\
		\mathcal{L}
		& \geq    \alpha_1 e_1^2  + \frac{1}{k_\omega}e_2^2 + \frac{1}{k_v}e_3^2 - e_1(\frac{\sqrt{2}\mu_1 }{2} e_2+ \mu_2  e_3),
		\end{align*}
		which implies
	\begin{align*}
	e\T \underbrace{\begin{bmatrix}
			\alpha_1 &   \frac{-\sqrt{2}}{4}\mu_1 &    \frac{-1}{2} \mu_2\\
			\frac{-\sqrt{2}}{4}\mu_1 & \frac{1}{k_\omega} &0 \\
			\frac{-1}{2} \mu_2 & 0 & \frac{1}{k_v}
			\end{bmatrix}}_{P_1} e\leq  \mathcal{L}(x) \leq e\T \underbrace{\begin{bmatrix}
			\alpha_2  &   \frac{\sqrt{2}}{4}\mu_1   &  \frac{1}{2} \mu_2\\
			\frac{\sqrt{2}}{4}\mu_1 & \frac{1}{k_\omega}  & 0 \\
			\frac{1}{2} \mu_2\ & 0 & \frac{1}{k_v}
			\end{bmatrix}}_{P_2}e.
	\end{align*}
	Using the fact
	$
	|x|_{\bar{\mathcal{A}}}^2
	=|\tilde{g}|_I^2 + \|\tilde{b}_a\|^2 = \|e\|^2,
	$
	one obtains the following inequalities:
	\begin{equation}
	\lambda_{\min}^{P_1} |x|_{\bar{\mathcal{A}}}^2 \leq \mathcal{L}(x) \leq  \lambda_{\max}^{P_2}  |x|_{\bar{\mathcal{A}}}^2.  \label{eqn:inequalities}
	\end{equation}
	Let $\psi_{\tilde{R}} := \psi_a(\tilde{R})$. One verifies that $\|\psi_{\tilde{R}}\|^2  \leq \frac{1}{2}\|I_3-\tilde{R}\|_F^2 \leq 4$ and $\|\psi_{\tilde{R}}\|^2  \leq \frac{1}{2} |\tilde{g}|_I^2$.  From (\ref{eqn:closed_loop}), one has
	\begin{align}
	\dot{\tilde{R}}~~ & = \tilde{R} (- k_\beta\psi_\omega +\hat{R}\tilde{b}_\omega )^\times , \label{eqn:dot_R}\\
	\dot{\tilde{p}}~~ & = \tilde{R}(-k_\beta\psi_v+\hat{p}^\times \hat{R} \tilde{b}_\omega + \hat{R}\tilde{b}_v ), \label{eqn:dot_pe} \\
	\dot{\psi}_{\tilde{R}} & = E(\tilde{R}) (- k_\beta\psi_\omega +\hat{R}\tilde{b}_\omega), \label{eqn:dot_psi_R}
	\end{align}
	where $E(\tilde{R}):=\frac{1}{2}(\tr(\tilde{R})-\tilde{R}\T)$. The arguments of $\psi$ have been omitted for simplicity, and $\psi_\omega, \psi_v$ are given by $\psi := [\psi_\omega\T, \psi_v\T]\T$.  From Lemma 2 in \cite{berkane2017hybrid2}, one has $\|E(\tilde{R})\|_F \leq \sqrt{3}$ and $v\T(I_3-E(\tilde{R}))v \leq \frac{1}{4}\|I_3-\tilde{R}\|_F^2 \|v\|^2\leq \frac{\sqrt{2}}{2}\|I_3-\tilde{R}\|_F \|v\|^2 \leq \frac{\sqrt{2}}{2} |\tilde{g}|_I \|v\|^2$ for all $v\in \mathbb{R}^3$. 	
	Define the constants $c_\omega := \sup_{t\geq 0} \|\omega(t)\|$ and $c_p: = \sup_{t\geq 0} \|p(t)\|$. Since $\tilde{b}_\omega$ is bounded both in the flow and jump sets, there exists a constant $c_{b_\omega}: = \sup_{(t,j)\succeq (0,0)}\|\tilde{b}_\omega(t,j)\|$. In view of (\ref{eqn:dot_R})-(\ref{eqn:dot_psi_R}), the time-derivative of the cross term $\mathfrak{X}: = z\T U \tilde{b}_a=-\mu_1 \tilde{b}_\omega\T \hat{R}\T \psi_{\tilde{R}}-\mu_2 \tilde{p}\T R \tilde{b}_v$ is obtained as
	\begin{align*}
	\dot{\mathfrak{X}} 
	&=  -\dot{z}\T U \tilde{b}_a - z\T U \dot{\tilde{b}}_a \\
	& = -\mu_1\tilde{b}_\omega\T (\dot{\hat{R}}\T \psi_{\tilde{R}} + \hat{R}\T \dot{\psi}_a(\tilde{R})) - \mu_2\tilde{b}_v\T (R\T \dot{\tilde{p}} + \dot{R}\T \tilde{p})   - z\T U \dot{\tilde{b}}_a  \\
	& =  \mu_1\tilde{b}_\omega\T (\omega  -\tilde{b}_\omega + k_\beta \psi_\omega)^\times \hat{R}\T \psi_{\tilde{R}}    -\mu_1 \tilde{b}_\omega \T \hat{R}\T E(\tilde{R}) (-k_\beta\psi_\omega +\hat{R}\tilde{b}_\omega ) + \mu_2\tilde{b}_v \T(\omega)^\times R\T \tilde{p}   \\
	& ~~~ - \mu_2\tilde{b}_v \T \hat{R}\T   (- k_\beta\psi_v  +\hat{p}^\times \hat{R} \tilde{b}_\omega + \hat{R}\tilde{b}_v )    + z\T U   \Gamma \Ad{g}^* \Ad{\tilde{g}^{-1}}^* \psi\|\\
	& \leq - \mu_1 \|\tilde{b}_\omega\|^2 - \mu_2\|\tilde{b}_v\|^2 +\mu_1 c_\omega \|\psi_{\tilde{R}}\|\|\tilde{b}_\omega\|  + \mu_1k_\beta  \|\tilde{b}_\omega\| \|\psi_\omega\|\|\psi_{\tilde{R}}\| \\
	& ~~~  +   \mu_1 \tilde{b}_\omega\T \hat{R}\T(I_3-E(\tilde{R})) \hat{R}\tilde{b}_\omega   + \mu_1 k_\beta \|\tilde{b}_\omega\| \|E(\tilde{R})\|_F \|\psi_\omega\|  + \mu_2 c_\omega  \|\tilde{b}_v\| \|\tilde{p}\|   \\
	& ~~~+ \mu_2 k_\beta \|\tilde{b}_v\| \|\psi_v\|+ \mu_2 c_{b_\omega}\|\tilde{b}_v\|\|\tilde{p}\|  + \mu_2 c_p \|\tilde{b}_v\| \|\tilde{b}_\omega\|  +  k_{\Gamma} \|U\|_2 \|z\| \|\Ad{g}^*\|_F \|\Ad{\tilde{g}^{-1}}^* \psi\| \\
	& \leq - \mu_1 \|\tilde{b}_\omega\|^2 - \mu_2\|\tilde{b}_v\|^2 +  2\sqrt{\alpha_4}  \mu_1k_\beta  \|\tilde{b}_\omega\| |\tilde{g}|_I + {\textstyle \frac{\sqrt{2}}{2}}\mu_1 c_\omega  |\tilde{g}|_I \|\tilde{b}_\omega\|   \\
	& ~~~  + {\textstyle\frac{\sqrt{2}}{2}}  \mu_1 c_{b_\omega} |\tilde{g}|_I \|\tilde{b}_\omega\|    +  \sqrt{3\alpha_4}\mu_1 k_\beta \|\tilde{b}_\omega\|  |\tilde{g}|_I + \mu_2 c_\omega  \|\tilde{b}_v\| |\tilde{g}|_I  +\sqrt{\alpha_4}  \mu_2 k_\beta \|\tilde{b}_v\| |\tilde{g}|_I  \\
	& ~~~+ \mu_2 c_{b_\omega}\|\tilde{b}_v\||\tilde{g}|_I + \mu_2 c_p \|\tilde{b}_v\| \|\tilde{b}_\omega\| + k_{\Gamma} c_g\sqrt{\alpha_5} (u_1 + u_2)|\tilde{g}|_I ^2 \\
	& \leq - \mu_1 \|\tilde{b}_\omega\|^2 - \mu_2\|\tilde{b}_v\|^2    + \mu_1({\textstyle \frac{\sqrt{2}}{2}} c_\omega + 2\sqrt{\alpha_4} k_\beta   +{\textstyle\frac{\sqrt{2}}{2}}   c_{b_\omega}+  \sqrt{3\alpha_4} k_\beta ) |\tilde{g}|_I \|\tilde{b}_\omega\|    \\
	& ~~~  + \mu_2 ( c_\omega    +\sqrt{\alpha_4}   k_\beta +  c_{b_\omega}) \|\tilde{b}_v\| |\tilde{g}|_I  + \mu_2 c_p \|\tilde{b}_v\| \|\tilde{b}_\omega\| + k_{\Gamma} c_g\sqrt{\alpha_5} (u_1 + u_2)|\tilde{g}|_I ^2 ,
	\end{align*} 	
	where $k_\Gamma:= \|\Gamma\|_F, c_g:=\|\Ad{g}^*\|_F= \sqrt{6+2c_p^2}$, and the following facts have been used: $|\tilde{g}|_I^2 = \|I_3-\tilde{R}\|^2+ \|\tilde{p}\|^2$, $\|\hat{p}\| = \|\tilde{R}\T(p- \tilde{p}) \| \leq c_p+ \|\tilde{p}\|$, $\|\psi_\omega\|^2+\|\psi_v\|^2 \leq \alpha_4 |\tilde{g}|_I^2$, $\|z\| = \|\psi_{\tilde{R}}\| + \|\tilde{p}\| \leq |\tilde{g}|_I$, $\|U\|_2 \leq  (\mu_1 + \mu_2)$ and Eq. (\ref{eq3_theo1}).  Let $ c_1 := \frac{\sqrt{2}}{2}c_\omega + 2k_\beta  \sqrt{\alpha_4} + \frac{\sqrt{2} c_{b_\omega}}{2} +   k_\beta \sqrt{3\alpha_4}, c_2:=k_\beta \sqrt{\alpha_4}+ c_{b_\omega} + c_\omega, c_3: = k_{\Gamma}\sqrt{\alpha_5} c_g $. Then, the time-derivative of $\mathfrak{X}$ satisfies
	\begin{align}
	\dot{\mathfrak{X}}
	& \leq - \mu_1 \|\tilde{b}_\omega\|^2 - \mu_2 \|\tilde{b}_v\|^2 + \mu_1 c_1 \|\tilde{b}_\omega\| |\tilde{g}|_I + \mu_2 c_2 \|\tilde{b}_v\||\tilde{g}|_I   \nonumber\\
	& ~~~ + ( \mu_1 c_3 + \mu_2 c_3 )|\tilde{g}|_I ^2      + \mu_2 c_p\|\tilde{b}_v\| \|\tilde{b}_\omega\|, \label{eqn:derivative_X}
	\end{align} 	
	Consequently, in view of (\ref{eqn:dot_V}) and (\ref{eqn:derivative_X}), one obtains
	\begin{align*}
	\dot{\mathcal{L}}
	& \leq  - 2k_\beta \alpha_3 e_1^2  + ( \mu_1 c_3 + \mu_2 c_3 )e_1^2   - \mu_1 \|\tilde{b}_\omega\|^2 - \mu_2 \|\tilde{b}_v\|^2   \\
	& ~~~  + \mu_1 c_1 \|\tilde{b}_\omega\| e_1+ \mu_2 c_2 \|\tilde{b}_v\| e_1     + \mu_2 c_p\|\tilde{b}_v\| \|\tilde{b}_\omega\|   \\
	& = -e\T \begin{bmatrix}
		2k_\beta \alpha_3   - \mu_1 c_3 - \mu_2 c_3 &  -\frac{1}{2} \mu_1 c_1 &  -\frac{1}{2} \mu_2 c_2 \\
		-\frac{1}{2} \mu_1 c_1 &  \mu_1 & -\frac{1}{2} \mu_2 c_p\\
		-\frac{1}{2} \mu_2 c_2 &-\frac{1}{2} \mu_2 c_p  &  \mu_2
		\end{bmatrix} e \\
	& = -e_{12}\T P_{31} e_{12}   -e_{13}\T  P_{32} e_{13}  - e_{23}\T  P_{33} e_{23},
	\end{align*} 	
	where $e_{ij} := [e_i,e_j]\T, i,j \in \{1,2,3\}$ with $e_i$ denoting the $i$-th elements of $e$, and
	\begin{align*}
	P_{31} &:= \begin{bmatrix}
	k_\beta \alpha_3  - \mu_1 c_3 & -\frac{1}{2} \mu_1 c_1 \\
	-\frac{1}{2} \mu_1 c_1 &  \frac{1}{2}\mu_1
	\end{bmatrix}, 
	P_{32}  := \begin{bmatrix}
	k_\beta \alpha_3 -   \mu_2 c_3 & -\frac{1}{2} \mu_2 c_2 \\
	-\frac{1}{2} \mu_2 c_2 &  \frac{1}{2} \mu_2
	\end{bmatrix}, P_{33} := \begin{bmatrix}
	\frac{1}{2} \mu_1 & -\frac{1}{2} \mu_2 c_p \\
	-\frac{1}{2} \mu_2 c_p & \frac{1}{2} \mu_2
	\end{bmatrix}.
	\end{align*}
	To guarantee that the matrices $P_1,P_2,P_{31},P_{32}$ and $P_{33}$ are positive definite, the $\mu_1$ and $\mu_2$ are chosen as follows:
	\begin{align*}
		\mu_1 & < \frac{2\sqrt{\alpha_1}}{\sqrt{k_\omega}}, \quad~~ \mu_2   < \frac{\sqrt{2\alpha_1}}{ \sqrt{k_v}}   ~~~~~~~~~~ \text{ for } P_1, P_2>0	\\
		\mu_1  &< \frac{k_\beta\alpha_3}{c_3 + \frac{1}{4}c_1^2} ,  \quad
		\mu_2   < \frac{k_\beta\alpha_3}{c_3 + \frac{1}{4}c_2^2}   ~~~~~~\text{ for } P_{31}, P_{32}>0 \\
		\mu_2 & <  \frac{1}{c_p^2} \mu_1  ~~~~~~~~~~~~~~~~~~~~~~~~~~~~~~~~\text{ for } P_{33}>0 
	\end{align*}
		which are equivalent to
	\begin{align*}
	&0 < \mu_1 < \min\left\{\frac{2\sqrt{\alpha_1}}{\sqrt{k_\omega}} , \frac{k_\beta\alpha_3}{c_3 + \frac{1}{4}c_1^2} \right\} \\
	&0 < \mu_2 < \min\left\{\frac{\sqrt{2\alpha_1}}{ \sqrt{k_v}} ,\frac{k_\beta\alpha_3}{c_3+ \frac{1}{4}c_2^2} , \frac{1}{c_p^2} \mu_1 \right\}.
	\end{align*}
	One concludes that
	\begin{equation}
	\dot{\mathcal{L}}(x) \leq -\lambda_F \mathcal{L}(x),\quad  x\in \mathcal{F}_c, \label{eqn:decrease_L_F}
	\end{equation}
	with $\lambda_F : =\min\{\lambda_{\min}^{P_{31}},\lambda_{\min}^{P_{32}},\lambda_{\min}^{P_{33}}\}/\lambda_{\max}^{P_2}$. On the other hand, from (\ref{eqn:dot_V}) and (\ref{eqn:decrease_V}), it is clear that $\tilde{b}_a$ is bounded in the flow and jump sets. Hence, there exists a constant $ c_{b_a}: = \sup_{(t,j)\succeq (0,0)}\|\tilde{b}_a(t,j)\|$. Let $g_q=\mathcal{T}(R_q,p_q)\in \mathbb{Q}$. Using the facts: $\tilde{R}^+ = \tilde{R}R_q$ and $\tilde{p}^+ = \tilde{p} + \tilde{R}p_q$, one has
	\begin{align*}
	\|z^+ - z\| 
	& = \left\|\begin{bmatrix}
		(\hat{R}^+)\T \psi_a(\tilde{R}^+)   - \hat{R} \T\psi_a(\tilde{R})\\
		R\T\tilde{p}^+ - R\T\tilde{p}
		\end{bmatrix}\right\| \\
	& \leq \|R_q \psi_a(\tilde{R}R_q)   -  \psi_a(\tilde{R})\|  +\| \tilde{p}^+ -  \tilde{p} \|  \leq 4+ \|p_q\| ,
	\end{align*}	
	which implies
	\begin{align*}
	\mathcal{L}(x^+)- \mathcal{L}(x)  &= V(x^+) - V(x)  + (z^+)\T U \tilde{b}_a  - z\T U \tilde{b}_a \\
	& \leq -\delta + (\mu_1 + \mu_2)c_4,
	\end{align*}
	where $c_4: = (4+  \|p_q\|  ) c_{b_a}$.  Choosing $\mu_1$ and $\mu_2$ such that
	\begin{align*}
	0 < \mu_1 < \min \left\{\frac{2\sqrt{\alpha_1}}{\sqrt{k_\omega}} , \frac{k_\beta\alpha_3}{c_3 + \frac{1}{4}c_1^2}, \frac{\delta}{2c_4} \right\}, \\
	0 < \mu_2 < \min\left\{\frac{\sqrt{2\alpha_1}}{ \sqrt{k_v}} ,\frac{k_\beta\alpha_3}{c_3+ \frac{1}{4}c_2^2} , \frac{1}{c_p^2} \mu_1, \mu_1  \right\},
	\end{align*}
	there exists a constant $\delta^*: = \delta - 2\mu_1 c_4 > 0$ such that one has
	\begin{equation}
	\mathcal{L}(x^+) \leq  \mathcal{L}(x)  - \delta^*,\quad  x\in \mathcal{J}_c.  \label{eqn:decrease_L_J}
	\end{equation}
	In view of (\ref{eqn:inequalities}), (\ref{eqn:decrease_L_F}) and (\ref{eqn:decrease_L_J}), one concludes
	\begin{align*}
	&\lambda_{\min}^{P_1} |x|_{\bar{\mathcal{A}}}^2 \leq \mathcal{L}(x) \leq  \lambda_{\max}^{P_2}  |x|_{\bar{\mathcal{A}}}^2, \quad, x\in \mathcal{S}\\
	&\dot{\mathcal{L}}(x) \leq -\lambda_F \mathcal{L}(x),\quad  x\in \mathcal{F}_c, \\
	&\mathcal{L}(x^+)\leq \exp(-\lambda_J)\mathcal{L}(x), \quad x\in \mathcal{J}_c,
	\end{align*}
	 with $\lambda_J: = -\ln (1-\delta^*/\mathcal{L}(0,0))$. Consequently, one obtains
	\begin{equation}
	\mathcal{L}(t,j) \leq \exp(-2\lambda (t+j)) \mathcal{L}(0,0) \label{eqn:L_t_j},
	\end{equation}
	where $\lambda: = \frac{1}{2}\min\{\lambda_F,\lambda_J\}$. From (\ref{eqn:inequalities}) and (\ref{eqn:L_t_j}), for each $(t,j)\in \dom{x}$ one has
	\begin{equation}
	|x(t,j)|_{\bar{\mathcal{A}}}\leq  k\exp\left(-\lambda (t+j) \right)|x(0,0)|_{\bar{\mathcal{A}}},
	\end{equation}
	where $k: = \sqrt{\scalemath{0.8}{{\lambda_{\max}^{P_2}}/{\lambda_{\min}^{P_1}}}}$. Since $\mathcal{F}_c \cup \mathcal{J}_c = \mathcal{S}$, the number of jumps if finite and there is no finite escape-time, one concludes that the solution to the hybrid system $\mathcal{H}$ is complete as per  Proposition 2.10 in \cite{goebel2012hybrid}.
	This completes the proof.

\subsection{Proof of Lemma \ref{lemma:Q}}
 	From the definition of matrix $\mathbb{A}$ in (\ref{eqn:definition_A}), one can show that
	\begin{align*}
	\mathbb{A} &=\sum_{i=1}^n k_i r_i r_i\T = \begin{bmatrix}
	\sum_{i=1}^{n_1} k_i p_i^\mathcal{I}(p_i^\mathcal{I})\T + \sum_{j=1}^{n_2}k_{j+n_1} v_j^\mathcal{I}(v_j^\mathcal{I})\T & \sum_{i=1}^{n_1} k_i p_i^\mathcal{I} \\
	\sum_{i=1}^{n_1} k_i (p_i^\mathcal{I})\T & \sum_{i=1}^{n_1} k_i 
	\end{bmatrix} ,
	\end{align*}
	which implies:
	\begin{align*}
	A &:= \sum_{i=1}^{n_1} k_i p_i^\mathcal{I}(p_i^\mathcal{I})\T + \sum_{j=1}^{n_2}k_{j+n_1} v_j^\mathcal{I}(v_j^\mathcal{I})\T, \\
	b &:= \sum_{i=1}^{n_1} k_i p_i^\mathcal{I}, \qquad 
	d  := \sum_{i=1}^{n_1} k_i.
	\end{align*}	
	From Assumption \ref{assump:1}, there exists at least one gain parameter $k_i > 0$ for all $i=1,2,\cdots,n_1$. Hence, one has $d >0$. Let $\alpha_i = k_i/(\sum_{i=1}^{n_1}k_i)$, and it is easy to verify that
	$
	p^{\mathcal{I}}_c =   \sum_{i=1}^{n_1} \alpha_i p^{\mathcal{I}}_i =bd^{-1}.
	$
	In view of (\ref{eqn:p_c}) and (\ref{eqn:definition_A}), one has
	\begin{align*}
	Q &= A - bb\T d^{-1} \nonumber  \\
	& =  \sum_{i=1}^{n_1} k_ip_i^{\mathcal{I}}\left(p_i^{\mathcal{I}}\right)\T  + \sum_{i=1}^{n_2} k_{j+n_1}v_i^{\mathcal{I}}\left(v_i^{\mathcal{I}}\right)\T   -  \left(\sum_{i=1}^{n_1} k_ip_i^{\mathcal{I}}\right)\left(\sum_{i=1}^{n_1} k_ip_i^{\mathcal{I}}\right)\T \left(\sum_{i=1}^{n_1} k_i \right)^{-1}  \\
	& =  \sum_{i=1}^{n_1} k_ip_i^{\mathcal{I}}\left(p_i^{\mathcal{I}}\right)\T   -   \sum_{i=1}^{n_1} k_i   p_c^{\mathcal{I}}\left(p_c^{\mathcal{I}}\right)\T + \sum_{i=1}^{n_2} k_{j+n_1}v_i^{\mathcal{I}}\left(v_i^{\mathcal{I}}\right)\T.
	\end{align*}
	Substituting $p_i^{\mathcal{I}} = \bar{v}_i^{\mathcal{I}}  + p_c^{\mathcal{I}}$ in (\ref{eqn:p_c}) and $\sum_{i=1}^{n_1} k_i \bar{v}_i^{\mathcal{I}} = 0$, one obtains
	\begin{align*}
	Q &   =  \sum_{i=1}^{n_1} k_i\left(\bar{v}_i^{\mathcal{I}}  + p_c^{\mathcal{I}}\right)\left(\bar{v}_i^{\mathcal{I}}  + p_c^{\mathcal{I}}\right)\T -   \sum_{i=1}^{n_1} k_i   p_c^{\mathcal{I}}\left(p_c^{\mathcal{I}}\right)\T  + \sum_{i=1}^{n_2} k_{j+n_1}v_i^{\mathcal{I}}\left(v_i^{\mathcal{I}}\right)\T    \nonumber \\
	& =  \sum_{i=1}^{n_1} k_i(\bar{v}_i^{\mathcal{I}})(\bar{v}_i^{\mathcal{I}} )\T + \sum_{i=1}^{n_2} k_{j+n_1}v_i^{\mathcal{I}}(v_i^{\mathcal{I}})\T.
	\end{align*}
	It is straightforward to verify that $Q$ is positive semi-definite from Assumption \ref{assump:1}. Then, one can further show that
	\begin{align}
	\bar{Q}   
	&  = \frac{1}{2}  \sum_{i=1}^{n_1} k_i  ((\bar{v}_i^{\mathcal{I}})^\times )^2 +  \frac{1}{2}  \sum_{j=1}^{n_2} k_{j+n_1} ((v_j^{\mathcal{I}})^\times)^2,
	\end{align}
	which is positive definite from lemma 2 in \cite{tayebi2013inertial}.
	This completes the proof.

\subsection{Proof of Lemma \ref{lemma:Delta}} \label{app:proof_lemma_7}
Let $(\lambda_v^Q, v)$ be the pair of eigenvalue and eigenvector of the matrix $Q$, \textit{i.e.}, $Qv = \lambda_v^Q  v$. Then, one has
\begin{equation}
\Delta_Q(u_q,v) = \tr(Q) - u_q\T Q u_q- 2\lambda_v^Q(1 - (u_q\T v)^2). \label{eqn: Delta_v}
\end{equation} 	
\begin{itemize} 		
	\item [1)] For the sake of analysis, let $\mathbb{E}(Q)=\{v_1,v_2,v_3\}$ with $v_1,v_2,v_3$ denoting the three orthogonal eigenvector of matrix $Q$.
	Using the fact $\mathbb{E}(Q)\subseteq \mathbb{U}$, for any $v\in \mathcal{E}(Q) $ one has
	\begin{align*}
	\max_{u_q\in \mathbb{U}} \Delta_Q(u_q,v) \geq  \max_{u_q\in \mathbb{E}(Q)} \Delta_Q(u_q,v).
	\end{align*}
	\begin{itemize}
		\item [a)] For the case $Q= \lambda_1^Q I_3$, for any $v\in \mathcal{E}(Q)$ and $u_q\in \mathbb{E}(Q)$, from (\ref{eqn: Delta_v}) one has
		\begin{align*}
		\Delta_Q(u_q,v) &= \tr(Q) - u_q\T Q u_q- 2\lambda^Q(1 - (u_q\T v)^2)  \\
		&=  2\lambda_1^Q   (u_q\T v)^2 , \quad q=1,2,3.
		\end{align*}
		Using the fact $\sum_{q=1}^3(u_q\T v)^2 = 1$ for any $v\in \mathbb{S}^2$, one has
		$
		\max_{u_q\in \mathbb{U}} (u_q\T v)^2 \geq \frac{1}{3}
		$
		which implies
		$
		\max_{u_q\in \mathbb{U}}u_q\T W_v u_q \geq \frac{2}{3}  \lambda_1^Q.
		$
		Then, for any $v\in \mathcal{E}(Q)$, one can show that
		\begin{align}
		\Delta_Q^* = \min_{v\in \mathcal{E}(Q)} \max_{u_q\in \mathbb{U}}\Delta(u_q,v)  \geq   \frac{2}{3}  \lambda_1^Q.  \label{eqb:solution_delta_U1}
		\end{align}
		
		\item [b)] For the case $\lambda_1^Q = \lambda_2^Q \neq \lambda_3^Q>0$, without loss of generality let $\mathcal{E}(Q) = \{v_3\}\cup \mathcal{S}_{12} $ with $\mathcal{S}_{12} := \text{span}\{v_1,v_2\}\cap \mathbb{S}^2$. Then, if $v = v_3$, from (\ref{eqn: Delta_v}) one has
		\begin{align*}
		\Delta(u_q,v) 	& = \tr(Q)  - u_q\T Qu_q - 2\lambda_3^Q (1 - (u_q\T v_3 )^2)  \\
		& = \begin{cases}
		\tr(Q)   - \lambda_{1}^Q  - 2\lambda_{3}^Q, &  u_q \in \{v_1, v_2\}\\
		\tr(Q) -  \lambda_{3}^Q,  & u_q=v_3
		\end{cases}.
		\end{align*}
		Consequently, one has
		\begin{align}
		\max_{u_q\in \mathbb{U}} \Delta(u_q,v) =  2\lambda_{12}^Q , \quad  v= v_3 \label{eqn:vinv3}.
		\end{align}
		If $v \in \mathcal{S}_{12}$, one has
		\begin{align*}
		\Delta(u_q,v) 	& = \tr(Q)  - u_q\T Qu_q - 2\lambda^Q_{1} (1 - (u_q\T v )^2)  \\
		& = \begin{cases}
		\lambda_3^Q -\lambda_{1}^Q  + 2\lambda_{1}^Q (u_q\T v )^2 & u_q \in  \{v_1, v_2\} \\
		0,  & u_q = v_3
		\end{cases}.
		\end{align*}
		Using the fact that $u_1$ and $u_2$ are orthogonal, one has $\sum_{q=1}^2(u_q\T v)^2 = 1$ for any $v\in \mathcal{S}_{12}$. 
		Then, one has
		\begin{align}
		\max_{u_q\in \mathbb{U}} \Delta(u_q,v)  \geq \lambda_3^Q   , \quad   v\in \mathcal{S}_{12} \label{eqn:vinv12}.
		\end{align}
		Hence, in view of (\ref{eqn:vinv3}) and (\ref{eqn:vinv12}), one obtains
		\begin{align}
		\min_{v\in \mathcal{E}(Q)} \max_{u_q\in \mathbb{U}} \Delta(u_q,v)   \geq  \min\left\{2\lambda^Q_{12}, \lambda_3^Q \right\}.  \label{eqb:solution_delta_U2}
		\end{align}
		
		\item [c)] If matrix $Q$ has three distinct eigenvalues, \ie, $\lambda_i^Q \neq \lambda_j^Q, i\neq j =1,2,3$, for each $v_s\in \mathcal{E}(Q)$, one has
		\begin{align*}
		\Delta(u_q,v_s)
		&= \tr(Q)  - u_q\T Qu_q - 2v_s\T Q v_s (1 - (u_q\T v_s)^2)  \\
		& = \begin{cases}
		\tr(Q)  - u_q\T Qu_q - 2v_s\T Q v_s & q\neq s \\
		\tr(Q)  - u_s\T Qu_s   & q= s \\
		\end{cases}.
		\end{align*}
		Since $\tr(Q)  - u_q\T Qu_q \geq \tr(Q)  - u_q\T Qu_q - 2v_s\T Q v_s$, one obtains
		\begin{align*}
		\max_{u\in \mathbb{U}} \Delta(u_q,v_s)  = \tr(Q) -   \lambda_s^Q  ,\quad  \forall v_s\in \mathcal{E}(Q).
		\end{align*}
		Hence, one can further show that
		\begin{align}
		\min_{v\in \mathcal{E}(Q)}\max_{u\in \mathbb{U}} \Delta(u_q,v)   = \tr(Q) -   \lambda_{\max}^Q.\label{eqb:solution_delta_U3}
		\end{align}
		From (\ref{eqb:solution_delta_U1}), (\ref {eqb:solution_delta_U2}) and (\ref {eqb:solution_delta_U3}), one concludes (\ref{eqb:solution_delta_U}).
	\end{itemize}
	\item [2)] If $\tr(Q) -2\lambda_{\max}^Q> 0$, let $u_1,u_2,u_3$ be any orthonormal basis in $\mathbb{R}^3$ and choose $ \{u_1,u_2,u_3\}\subseteq \mathbb{U}$. Then, for any $v\in \mathcal{E}(Q), u_q\in \mathbb{U}$, one has
	\begin{align*}
	\Delta(u_q,v)   &= \tr(Q) - u_q\T Q u_q- 2\lambda^Q_v(1 - (u_q\T v)^2).
	\end{align*}
	Using the facts $\sum_{q=1}^3(u_q\T v)^2 = 1$ for any $v\in \mathbb{S}^2$ and $ \sum_{q=1}^3 u_q\T Q u_q = \tr(Q)$, one has
	\begin{align*}
	\frac{1}{3}\sum_{q=1}^3 \Delta(u_q,v)   
	&=  \frac{1}{3}\left(3\tr(Q) - \sum_{q=1}^3 u_q\T Q u_q - 2\lambda_v^Q \left(3-\sum_{q=1}^3(u_q\T v)^2\right)\right)  \\
	&=  \frac{1}{3}\left(2\tr(Q)  - 4\lambda_v^Q  \right),
	\end{align*}
	which implies that
	$
	\max_{u_q\in \mathbb{U}} \Delta(u_q,v)  \geq \frac{2}{3}\left(\tr(Q)  - 2\lambda_v^Q  \right) .
	$
	For any $v\in \mathcal{E}(Q)$, one can show that
	\[
	\min_{v\in \mathcal{E}(Q)}\max_{u_q\in \mathbb{U}} \Delta(u_q,v)  \geq \frac{2}{3}\left(\tr(Q)  - 2\lambda_{\max}^Q  \right) .
	\]
	which gives (\ref{eqn:Delta_Q2}). 
\end{itemize}
This completes the proof.
 
\subsection{Proof of Lemma \ref{lemma: tr_and_psi}}
	From the definition of $\mathbb{A}$ in (\ref{eqn:definition_A}) as per Lemma \ref{lemma:Q}, one obtains
	\begin{align*}
	 \sum_{i=1}^{n} k_i \|r_i-g^{-1}r_i\|^2
	& =  \tr\left(  (I_4-g^{-1})\mathbb{A}  (I_4-g^{-1})\T\right)  \\
	& = \tr\left(g\T g  (I_4-g^{-1})\mathbb{A}  (I_4-g^{-1})\T\right)  \\
	& =   \tr\left(  (I_4 - g)\mathbb{A} (I_4 -g)\T\right)  	,
	\end{align*}
	where property (\ref{eqn:property2}) and the fact $(I_4-g^{-1}), (I_4-g^{-1})\mathbb{A} \in \mathcal{M}_0$ have been used. Furthermore, using the fact $(\bar{g}^{-1} g^{-1}r_i) \wedge (\bar{g}^{-1} r_i)  =\Ad{\bar{g}^{-1}}^* ((g^{-1} r_i)\wedge r_i)$, one verifies that
	\begin{align*}
	&\psi\left(\mathbb{P}((I_4 - g^{-1})\mathbb{A})\right)
	=\sum_{i=1}^{n} k_i  \psi\left((I_4  - g^{-1})   r_ir_i\T \right)   = \frac{1}{2} \sum_{i=1}^{n} k_i (g^{-1}r_i  ) \wedge r_i \\
	 &\Ad{g_1}^*  \sum_{i=1}^{n} k_i  (\bar{g}g^{-1}r_i  ) \wedge r_i  =   \sum_{i=1}^{n} k_i  (g^{-1}r_i  ) \wedge (\bar{g}^{-1} r_i).
	\end{align*}
	This completes the proof.

\subsection{Proof of Lemma \ref{lemma:gradient_on_SE_3}}
	Given $g = \mathcal{T}(R,p)$， $Q=A-bb\T d^{-1}$, from (\ref{eqn:definition_A}) one can verify that
	\begin{align*}
		\mathcal{U}(g)
		& =\frac{1}{2} \tr\left( \begin{bmatrix}
		I_3 - R & -p\\
		0 & 0
		\end{bmatrix} \begin{bmatrix}
		A & b\\
		b\T & d
		\end{bmatrix}\begin{bmatrix}
		I_3 - R & -p\\
		0 & 0
		\end{bmatrix} \T\right) \\
		& = \frac{1}{2} \tr\left(  
		(I_3 - R)A (I_3 - R)\T- 2(I_3 - R)bp\T+dpp\T 
		 \right) \\
		& =  \tr\left(Q(I_3 - R)\right) + \frac{1}{2} d\left\|p - (I_3-R)bd^{-1} \right\|^2 .
	\end{align*}
	The gradient $\nabla_g\mathcal{U}(g) $ can be computed using the differential of $\mathcal{U}(g)$ in an arbitrary tangential direction $gX\in T_gSE(3)$ with some $X\in \mathfrak{se}(3)$
	\begin{align}
	d \mathcal{U}(g) \cdot gX
	&= \langle \nabla_g \mathcal{U}(g),gX\rangle_g 
	  = \langle\langle g^{-1}\nabla_g \mathcal{U}(g),X\rangle\rangle \label{eqn:gradient_on_metric}.
	\end{align}
	On the other hand, from the definition of the tangent map, one has
	\begin{align}
	d \mathcal{U}(g) \cdot gX
	& = \tr \left(-gX\mathbb{A}(I_4 -g)\T \right) \nonumber \\
	& =  \langle\langle g\T(g - I_4  ) \mathbb{A}, X\rangle\rangle \nonumber \\
	& =  \langle\langle \mathbb{P}(g\T(g- I_4 ) \mathbb{A}), X\rangle\rangle  \nonumber\\
	& =  \langle\langle \mathbb{P}(g^{-1}(g - I_4 ) \mathbb{A}), X\rangle\rangle  \nonumber \\
	& =  \langle\langle \mathbb{P}((I_4 -g^{-1}) \mathbb{A}), X\rangle\rangle \label{eqn:d_L_A},
	\end{align}
	where the fact $(I_4 -g^{-1})\mathbb{A} \in \mathcal{M}_0$ and property (\ref{eqn:property1}) have been used. Hence, in view of (\ref{eqn:gradient_on_metric}) and (\ref{eqn:d_L_A}), one can verify (\ref{eqn:gradient-g}).	
	
	Using the fact $ g\mathbb{P}((I_4 -g^{-1})\mathbb{A})  \in \mathcal{M}_0$, the identity $\nabla_g \mathcal{U}(g) = 0$ implies that $(I_4 -g^{-1})\mathbb{A} = \mathbb{A}(I_4 -g^{-1})\T$, which can be further reduced as $\mathbb{A}g\T = g\mathbb{A}$. Applying the matrix decomposition (\ref{eqn:A_dcomposition}), one obtains
	\begin{align*}
	&\begin{bmatrix}
	Q & 0 \\
	0 & d
	\end{bmatrix} \begin{bmatrix}
	R & p - (I_3-R)bd^{-1}\\
	0 & 1
	\end{bmatrix}\T  = \begin{bmatrix}
	R & p - (I_3-R)bd^{-1}\\
	0 & 1
	\end{bmatrix} \begin{bmatrix}
	Q & 0 \\
	0 & d
	\end{bmatrix}.
	\end{align*}
	Consequently, one can further deduce that
	\begin{align}
	RQ = QR\T , \quad
	p = (I_3-R)bd^{-1}. \label{eqn: Randp}
	\end{align}
	Following the proof of Lemma 2 in \cite{mayhew2011synergistic}, one has the solution of (\ref{eqn: Randp}) as
	\begin{align*}
	&\{R = I_3, p=0\} \text{ or }  
	 \{R = \mathcal{R}_a(\pi,v), p =(I_3 - \mathcal{R}_{a}(\pi, v))bd^{-1}\},
	\end{align*}
	where $v \in \mathcal{E}(Q)$, which implies (\ref{eqn:set_Psi_U}). This completes the proof.
%
\subsection{Proof of Lemma \ref{lemma:conditions_exp}}
	Let $p_e: = p-(I_3-R)bd^{-1}$. 
	For the case $(I_3-R)bd^{-1}=0$, one has
	\[
	\|I_3-R\|_F^2 + \|p_e\|^2 = |g|_I ^2.
	\]
	For the case $(I_3-R)bd^{-1}\neq 0$, there exists a positive scalar $0 <\varrho_1(R) < \|bd^{-1}\|$ such that $\|(I_3-R)bd^{-1}\| = \varrho_1  \|I_3-R\|_F$. Let $\phi_1: = < p, (I_3-R)bd^{-1}>$ with $<,>$ denoting the angle of two vectors. Then, one has
	\begin{align*}
		&\|I_3-R\|_F^2 + \|p_e\|^2   \\
		& = (1+\varrho_1^2)\|I_3-R\|_F^2 + \|p\|^2 -2\varrho_1\cos\phi_1 \|p\|\|I_3-R\|_F \\
 		& = \begin{bmatrix}
		\|p\| & \|I_3-R\|_F
		\end{bmatrix} \underbrace{\begin{bmatrix}
		1 & -\varrho_1\cos\phi_1 \\
		-\varrho_1\cos\phi_1 & 1+\varrho_1^2
		\end{bmatrix}}_{\Theta_1}  \begin{bmatrix}
		\|p\| \\
		\|I_3-R\|_F
		\end{bmatrix} .
	\end{align*}
	which implies that 
	\begin{equation}
	 s_1 |g|_I^2  \leq \|I_3-R\|_F^2 + \|p_e\|^2 \leq s_2 |g|_I^2 ,
	\end{equation}
	where $s_1:=\min\{1,\lambda_{\min}^{\Theta_1}\}$ and $s_2:=\max\{1,\lambda_{\max}^{\Theta_1}\}$. Note that matrix $\Theta_1$ is always positive definite.
	Using the fact $\bar{Q} = \frac{1}{2}(\tr(Q)I_3-Q)$, by Assumption \ref{assump:1} one has $\lambda_{\min}^{\bar{Q}}>0$. For any $R\in SO(3)$ one verifies that
	\begin{equation}
	\frac{1}{2}\lambda_{\min}^{\bar{Q}} \|I_3 -R\|_F^2 \leq \tr(Q(I_3-R)) \leq \frac{1}{2} \lambda_{\max}^{\bar{Q}} \|I_3 -R\|_F^2.  \label{eqn:bounded_tr_Q}
	\end{equation}
	Then, from (\ref{eqn:definitinU})   one can show that
	\begin{align*}
	\mathcal{U}(g) \geq \frac{1}{2}\lambda_{\min}^{\bar{Q}} \|I_3 -R\|_F^2+ \frac{d}{2}\|p_e\|^2  \geq \alpha_1 |g|_I^2 ,  \\
	\mathcal{U}(g) \leq \frac{1}{2}\lambda_{\max}^{\bar{Q}} \|I_3 -R\|_F^2+ \frac{d}{2}\|p_e\|^2  \leq \alpha_2 |g|_I^2 ,
	\end{align*}
 	with 
 	\begin{align}
 	\alpha_1 := \min\left\{\frac{1}{2}\lambda_{\min}^{\bar{Q}} , \frac{1}{2}d\right\}s_1,\quad \alpha_2 := \max\left\{\frac{1}{2}\lambda_{\max}^{\bar{Q}}, \frac{1}{2}d\right\}s_2.
 	\end{align}
	Using the fact
	\begin{align*}
	(I_4 -g^{-1})\mathbb{A}
	&=\begin{bmatrix}
	I-  R\T & R\T p \\
	0 & 0
	\end{bmatrix}\begin{bmatrix}
	A & b\\
	b\T & d
	\end{bmatrix}  = \begin{bmatrix}
		(I-R\T)A + R\T p b\T & (I-R\T)b + R\T pd \\
		0 & 0
		\end{bmatrix} \\
	& =  \begin{bmatrix}
	(I_3-R\T)Q + R\T p_e b\T & R\T d p_e\\
	0 & 0
	\end{bmatrix} ,
	\end{align*}
	From (\ref{eqn:definiton_psi}), one verifies that
	\begin{align}
	\psi((I_4 -g^{-1})\mathbb{A})=\frac{1}{2}\begin{bmatrix}
	2\psi_a(QR) +  b^\times R\T p_e \\
	R\T   p_ed
	\end{bmatrix} \label{eqn:psi_with_A}.
	\end{align} 
	Hence, one can show that
	\begin{align*}
	& \|\psi((I_4 -g^{-1})\mathbb{A})\|^2 \\
	& = \frac{1}{4}\|2\psi_a(QR) + b^\times R\T p_e\|^2 +  \frac{1}{4} d^2 \|p_e\|^2 \\
	& =  \|\psi_a(QR) \|^2 +  \psi_a(QR)\T b^\times R\T p_e  +  \frac{1}{4}\| b^\times R\T p_e \|^2  + \frac{1}{4}d^2 \|p_e\|^2 \\
	& \geq \|\psi_a(QR) \|^2 -\|\psi_a(QR)\| \| b^\times R\T p_e \| +  \frac{1}{4}\| b^\times R\T p_e \|^2  +  \frac{1}{4}d^2 \|p_e\|^2 \\
	& = \|\psi_a(QR) \|^2 -   |\sin\phi_2|\|b\| \|\psi_a(QR)\|  \| p_e \|+    \frac{1}{4} s_1 \|b\|^2 \|p_e\|^2 +  \frac{1}{4}d^2\| p_e \|^2  \\
	& = \begin{bmatrix}
	\|\psi_a(QR) \|  &
	\| p_e \|
	\end{bmatrix}
	\Theta_2 \begin{bmatrix}
	\|\psi_a(QR) \| \\
	\| p_e \|
	\end{bmatrix},
	\end{align*}
	where
	\begin{align*}
	\Theta_2 := \begin{bmatrix}
	1 & -\frac{1}{2}|\sin\phi_2|\|b\| \\
	-\frac{1}{2}|\sin\phi_2|\|b\|  & \frac{1}{4}(\sin^2\phi_2\|b\|^2 + d^2)
	\end{bmatrix},  
	\end{align*}
	and we made use of the fact that
	$
	\| b^\times R\T p_e \|^2
	 =  p_e\T R(b\T b I_3 - bb\T) R\T p_e
	 = \sin^2\phi_2 \|b\|^2 \|p_e\|^2
	$
	with $\phi_2 := <b,R\T p_e>$. It is straightforward to verify that the matrix $\Theta_2$ is positive definite. Let $\lambda_{\min}^{\Theta_2}$ be the minimum eigenvalue of the matrix $\Theta_2$. One has
	\begin{align*}	
	 \|\psi((I_4 -g^{-1})\mathbb{A})\|^2   &= \frac{1}{4}\|2\psi_a(QR) +  b^\times R\T p_e\|^2 +  \frac{1}{4} d^2 \|p_e\|^2 \\
	& \geq \lambda_{\min}^{\Theta_2} (\|\psi_a(QR) \|^2 +   \| p_e \|^2) \\
	& = \lambda_{\min}^{\Theta_2} ( \vartheta(Q,R)\tr(\underline{Q}(I_3-R)) +   \| p_e \|^2)  \\
	& = \lambda_{\min}^{\Theta_2} \left( \frac{1}{2}\vartheta(Q,R)\lambda_{\min}^{\underline{Q}} \|I_3-R\|_F^2 +   \| p_e \|^2\right),
	\end{align*}
	where  $\underline{Q} := \tr(\bar{Q}^2)I_3 - 2\bar{Q}^2$, and we made use of the fact that $\|\psi_a(QR)\|^2 =\vartheta(Q,R)\tr(\underline{Q}(I_3-R))	$
	with $\vartheta(Q,R) := (1-\frac{1}{8}\|I_3-R\|_F^2 \cos^2\langle u,\bar{Q}u\rangle$  and $u\in \mathbb{S}^2$ denoting the axis of the rotation $R$ (see Lemma 2 in \cite{berkane2017hybrid}). For any $(R,p)\in \Upsilon$, one has $\tr((I-R)Q)-\tr((I-RR_q)Q)\leq \delta/2$, which follows $0<\vartheta(Q,R)\leq 1$. Hence, there exists a positive $\vartheta^*$ such that
	\begin{align*}
	\vartheta^* := \min_{g=\mathcal{T}(R,p)\in\Upsilon} \left (1-\frac{1}{8}\|I_3-R\|_F^2 \cos^2\langle u,\bar{Q}u\rangle\right)  .
	\end{align*}
	Therefore, one obtains
	\begin{align*}	
	\|\psi((I_4 -g^{-1})\mathbb{A})\|^2
	& \geq \lambda_{\min}^{\Theta_2} \left( \frac{1}{2}\vartheta^*\lambda_{\min}^{\underline{Q}} \|I_4 -R\|_F^2 +   \| p_e \|^2\right)   \geq \alpha_3 |g|_I ^2,
	\end{align*}
	where 
	\begin{equation}
	\alpha_3 := \lambda_{\min}^{\Theta_2} \min\left\{1,\frac{1}{2}\vartheta^*\lambda_{\min}^{\underline{Q}}\right\} s_1.
	\end{equation}
	Moreover, one can also show that
	\begin{align*}
	  \|\psi((I_4 -g^{-1})\mathbb{A})\|^2 
	&  = \frac{1}{4}\|2\psi_a(QR) +  b^\times R\T p_e\|^2 +  \frac{1}{4} d^2 \|p_e\|^2 \\
	& \leq 2\|\psi_a(QR) \|^2  +  \frac{1}{2}\| b^\times R\T p_e \|^2 +  \frac{1}{4}d^2 \|p_e\|^2 \\
	&  \leq 2\|\psi_a(QR) \|^2  +    \frac{1}{2}   \|b\|^2 \|p_e\|^2  +  \frac{1}{4}d^2\| p_e \|   \\
	&  \leq \lambda^{\underline{Q}}_{\max}\|I_3-R\|_F^2 + (\frac{1}{2}\|b\|^2 +  \frac{1}{4}d^2) \|p_e\|^2 \\
	&  = \alpha_4 |g|_I ^2.
	\end{align*}
	where \begin{equation}
	\alpha_4 := \max\left\{ (\frac{1}{2}\|b\|^2 +  \frac{1}{4}d^2),\frac{1}{2} \lambda_{\max}^{\underline{Q}}\right\}s_2.
	\end{equation}
	From (\ref{eqn:psi_with_A}),  one obtains
	\begin{align*}
	\Ad{g^{-1}}^* \psi((I_4 -g^{-1})\mathbb{A})
	& =\frac{1}{2} \begin{bmatrix}
	R & p^\times R \\
	0 & R
	\end{bmatrix}\begin{bmatrix}
	2\psi_a(QR) +  b^\times R\T p_e \\
	R\T   p_ed
	\end{bmatrix} \\
	& =\frac{1}{2} \begin{bmatrix}
	2R\psi_a(QR) + R b^\times R\T p_e +  d p^\times p_e  \\
	d p_e
	\end{bmatrix} \\
	& =\frac{1}{2} \begin{bmatrix}
		 2R\psi_a(QR) + (Rb+dp)^\times   p_e   \\
		 d p_e
		 \end{bmatrix} \\
	 & =\frac{1}{2} \begin{bmatrix}
	2R\psi_a(QR) + b^\times   p_e   \\
	d p_e
	\end{bmatrix},
	\end{align*}
	where the facts $dp = dp_e+(I_3-R)b$ and $p_e^\times p_e = 0$ have been used. Similarly, one obtains
	\begin{align*}
	  \|\Ad{g^{-1}}^* \psi((I_4 -g^{-1})\mathbb{A}) \|^2  &\leq  2\|\psi_a(QR)\| + \frac{1}{2} \|b^\times   p_e\|^2  +  \frac{1}{4}  d \|p_e\|^2 \\
	& \leq   \lambda_{\max}^{\underline{Q}} \|I_3-R\|_F^2 +  (\frac{1}{2} \|b\|^2+\frac{1}{4} d)\| p_e\|^2  \\
	&  \leq \alpha_4 |g|_I ^2.
	\end{align*}
	This completes the proof.

\subsection{Proof of Proposition \ref{pro:1}} \label{sec:pro1}
	Let $\tilde{g}=\mathcal{T}(\tilde{R},\tilde{p})$ with $\tilde{R}=\mathcal{R}_a(\theta,v)$, $\theta \in \mathbb{R}$ and $v\in \mathbb{R}^3$. In view of (\ref{eqn:definition_A}) and (\ref{eqn:definitinU1}), one obtains
	\begin{align}
	\mathcal{U}_1(\tilde{g})
	& = \tr((I_3-\tilde{R})Q)+ \frac{1}{2}d\|\tilde{p}-(I_3-\tilde{R})bd^{-1}\|^2  \label{eqn:U_g} \\
	&= 2(1-\cos\theta)v\T \bar{Q}v + \frac{1}{2}d\|\tilde{p}-(I_3-\tilde{R})bd^{-1}\|^2
	\end{align}
	where matrix $\bar{Q}$ is defined in Lemma \ref{lemma:Q}. Under Assumption \ref{assump:1}, one has $d>0$ and $\bar{Q}$ is positive definite. Consequently, $\mathcal{U}_1(\tilde{g})$ has a unique global minimum at $\tilde{g}=I_4$, \ie, $\mathcal{U}_1$ is a potential function on $SE(3)$. From (\ref{eqn:U_g}), for any $\tilde{g}=\mathcal{T}(\tilde{R},\tilde{p})\in \mathcal{X}_{\mathcal{U}_1}$ and $g_q=\mathcal{T}(R_q, p_q) \in \mathbb{Q}$,  one has
	\begin{align}
	\mathcal{U}_{1}(\tilde{g}g_q) &=  \tr((I_3-\tilde{R}R_q)Q)  + {  \frac{1}{2}} d\left\| p_q  - (I_3 - R_q )bd^{-1}  \right\|^2 \nonumber\\
	&= \tr((I_3-\tilde{R})Q)+ \tr(\tilde{R}(I_3-R_q)Q) \nonumber\\
	& =  \mathcal{U}_{1}(\tilde{g})- (1-\cos\theta^*)\Delta_Q(u_q,v) \label{Uggq}.
	\end{align}
	where $u_q\in \mathbb{U}, v\in \mathcal{E}(Q)$, $\Delta_Q(\cdot)$ is defined in (\ref{eqn:Delta_u_v}) in Lemma \ref{lemma:Delta}. We also made use of the following facts: $p_q = (I_3-R_q)bd^{-1}$, $\tilde{R}=\mathcal{R}_a(\pi,v)=2vv\T-I_3$, $R_q = \mathcal{R}_a(\theta^*, u_q)$. From \eqref{Uggq}, one can obtain
	\begin{align*}
	\mathcal{U}_{1}(\tilde{g}) - \min_{g_q\in \mathbb{Q}}\mathcal{U}_{1}(\tilde{g}g_q) &  = (1-\cos\theta^*)  \max_{u_q\in \mathbb{U}} \Delta_Q(u_q,v), \forall v\in \mathcal{E}(Q) \\
	&  \geq  (1-\cos\theta^*) \min_{v\in \mathcal{E}(Q)} \max_{u_q\in \mathbb{U}} \Delta_Q(u_q,v) \\
	& \textstyle =  (1-\cos\theta^*) \Delta^*_Q > \delta,
	\end{align*}
	which implies that $\mathcal{X}_{\mathcal{U}_{1}} \times \mathbb{R}^6 \times SE(3)\times \mathbb{R}^6 \times \mathbb{R}_{\geq 0} \subseteq \mathcal{J}_c$ from (\ref{eqn: definition_jump_set}). In view of \eqref{eqn:kenimatic_g} and \eqref{eqn:observer_design1}-\eqref{eqn:observer_design_sigma1}, one can write the hybrid closed-loop system as in (\ref{eqn:closed_loop}). The proof is completed by using Lemma \ref{lemma:gradient_on_SE_3}, Theorem 1 and the fact that $\mathcal{X}_{\mathcal{U}_{1}} \times \mathbb{R}^6 \times SE(3)\times \mathbb{R}^6 \times \mathbb{R}_{\geq 0} \subseteq \mathcal{J}_c$.

\subsection{Proof Theorem \ref{theo:Theorem_2}}\label{sec: Theorem2}
	As shown in Proposition \ref{pro:1}, for any $\underline{\tilde{g}}\in \mathcal{X}_{\mathcal{U}_2}$, one verifies that
	\begin{align*}
	\mathcal{U}_2(\underline{\tilde{g}}) -  \min_{g_q\in \mathbb{Q}}\mathcal{U}_2(g_c^{-1}{\tilde{g}}g_q g_c)
	&  =	\mathcal{U}_2(\underline{\tilde{g}}) -  \min_{g_q\in \mathbb{Q}}\mathcal{U}_2(\underline{\tilde{g}}g_c^{-1}g_q g_c)  \\
	&   = \mathcal{U}_{1}(\tilde{g}) - \min_{g_q\in \mathbb{Q}}\mathcal{U}_{1}(\tilde{g}g_q)
	> \delta,
	\end{align*}	
	which implies that $\mathcal{X}_{\mathcal{U}_{2}} \times \mathbb{R}^6 \times SE(3)\times \mathbb{R}^6 \times \mathbb{R}_{\geq 0} \subseteq \mathcal{J}_c'$.
	Let us consider the following real-valued function:
	\begin{equation}
	\mathcal{L}_R = \tr(Q(I-\underline{\tilde{R}})) + \frac{1}{2k_\omega} \tilde{b}_\omega  \T \tilde{b}_\omega - \bar{\mu}_1 \psi_a(\underline{\tilde{R}})\T \underline{\hat{R}} \tilde{b}_\omega,
	\end{equation}
	where $\bar{\mu}_1 > 0$. Let $\bar{e}_1 := [\|I_3-\underline{\tilde{R}}\|_F,~\| \tilde{b}_\omega\| ]\T$.  Following similar steps as in the proof of Theorem \ref{theo:Theorem_1}, it is clear that there exists a constant $ c_{b_\omega}: = \sup_{(t,j)\succeq (0,0)}\|\tilde{b}_\omega(t,j)\|$, and a constant $\bar{\mu} _1^*$ such that for all $ \bar{\mu} _1 <  \bar{\mu} _1^*$
	\begin{align}
	&\bar{c}_1 \|\bar{e}_1\|^2 \leq \mathcal{L}_R \leq \bar{c}_2 \|\bar{e}_1\|^2 , \label{eqn:c1_L_R_c2}\\
	&\dot{\mathcal{L}}_R \leq -\bar{c}_3 \mathcal{L} _R, \quad  x'\in \mathcal{F}_c' \label{eqn:dot_L_R},
	\end{align}
	for some positive constants $\bar{c}_1, \bar{c}_2$ and $\bar{c}_3$.  Let us consider the following real-valued function:
	\begin{equation}
	\mathcal{L}_p = \frac{d}{2} \underline{\tilde{p}} \T\underline{\tilde{p}} + \frac{1}{ k_v} \tilde{b}_v  \T \tilde{b}_v - \bar{\mu}_2 \underline{\tilde{p}}\T \underline{R}\tilde{b}_v.
	\end{equation}
	Let $\bar{e}_2 := [\|\underline{\tilde{p}}\|, \|\tilde{b}_v\|]\T$. It is straightforward to show that
	\begin{equation}
	\bar{e}_2\T \underbrace{\begin{bmatrix}
		\frac{d}{2} & -\frac{\bar{\mu}_2}{2} \\
		-\frac{\bar{\mu}_2}{2} & \frac{1}{k_v}
		\end{bmatrix}}_{H_1} \bar{e}_2 \leq \mathcal{L}_p \leq \bar{e}_2\T \underbrace{\begin{bmatrix}
		\frac{d}{2} & \frac{\bar{\mu}_2}{2} \\
		\frac{\bar{\mu}_2}{2} & \frac{1}{k_v}
		\end{bmatrix}}_{H_2} \bar{e}_2. \label{eqn:H1_L_P_H2}
	\end{equation}
	The time-derivative of $\mathcal{L}_p$ along the trajectories of (\ref{eqn:closed_tilde_p3}) and (\ref{eqn:dynamic_tilde_b_v3}) is obtained as
	\begin{align*}
	\dot{\mathcal{L}}_p 
	&=   d\underline{\tilde{p}} \T\dot{\underline{\tilde{p}}} + \frac{1}{ k_v} \tilde{b}_v  \T \dot{\tilde{b}}_v - \bar{\mu}_2 \dot{\underline{\tilde{p}}}\T R\tilde{b}_v
		- \bar{\mu}_2 \underline{\tilde{p}}\T R\dot{\tilde{b}}_v - \bar{\mu}_2 \underline{\tilde{p}}\T \dot{R}\tilde{b}_v \\
	&=  -\frac{d^2}{2}k_\beta   \underline{\tilde{p}}\T \underline{\tilde{p}} + d\underline{\tilde{p}} \T \underline{\tilde{R}}  \underline{\hat{p}}^{\times}\underline{\hat{R}} \tilde{b}_{\omega} + \bar{\mu}_2 k_\beta \frac{d}{2}  \underline{\tilde{p}}\T  R\tilde{b}_v   + \bar{\mu}_2 \tilde{b}_{\omega}\T \underline{\hat{R}} \T \underline{\hat{p}}^{\times} \underline{\tilde{R}}\T R\tilde{b}_v  -\bar{\mu}_2 \tilde{b}_v \T \tilde{b}_v   \\
	& ~~~	+ \frac{d}{2}  k_v  \bar{\mu}_2 \underline{\tilde{p}}\T R    \underline{\hat{R}}\T\underline{ \tilde{R}}\T\underline{ \tilde{p}}    - \bar{\mu}_2 \underline{\tilde{p}}\T R \omega^\times \tilde{b}_v \\
	&\leq  - \frac{d^2}{2}k_\beta \underline{\tilde{p}}\T \underline{\tilde{p}} -\bar{\mu}_2 \tilde{b}_v \T  \tilde{b}_v  + d\underline{\tilde{p}} \T \underline{\tilde{R}}  (\underline{\tilde{R}} \T\underline{p} - \underline{\tilde{R}} \T \underline{\tilde{p}})^{\times}\underline{\hat{R}} \tilde{b}_{\omega}   +\bar{\mu}_2 k_\beta \frac{d}{2} \|\underline{\tilde{p}}\| \|\tilde{b}_v \|  \\
	& ~~~+ \bar{\mu}_2 \tilde{b}_{\omega}\T  \underline{\hat{R}}  (\underline{\tilde{R}} \T\underline{p} - \underline{\tilde{R}} \T \underline{\tilde{p}})^{\times} \underline{\tilde{R}}\T R\tilde{b}_v 
		+ \bar{\mu}_2 \frac{d}{2}k_v\|\underline{\tilde{p}}\|^2 + \bar{\mu}_2 \|\omega\| \| \underline{\tilde{p}}\|  \|\tilde{b}_v \| \\
	&\leq  - \frac{d^2}{2}k_\beta \|\underline{\tilde{p}}\|^2 -\bar{\mu}_2 \|\tilde{b}_v\|^2  + d\bar{c}_p\|\underline{\tilde{p}}\| \| \tilde{b}_{\omega}\| + \bar{\mu}_2   k_\beta \frac{d}{2} \|\underline{\tilde{p}}\| \|\tilde{b}_v \| + \bar{\mu}_2 \bar{c}_p \|\tilde{b}_{\omega}\|\|\tilde{b}_v \|  \\
	& ~~~ + \bar{\mu}_2 c_{b_\omega} \|\underline{\tilde{p}}\| \|\tilde{b}_v\|
	+ \bar{\mu}_2 \frac{d}{2}k_v\|\underline{\tilde{p}}\|^2  + \bar{\mu}_2 c_\omega \| \underline{\tilde{p}}\|  \|\tilde{b}_v \|,
	\end{align*}
	where $\bar{c}_p:=\|\underline{p}\|=\|p-bd^{-1}\| $, and the following facts have been used: $\underline{\hat{R}}=\hat{R}, \underline{R}=R$ and $\underline{\hat{p}}=\underline{\tilde{R}}\T(\underline{p}-\underline{\tilde{p}})$.
	Let  $\bar{c}_4 := \max\{d,\bar{\mu}_2 \}$ and $ \bar{c}_5: =  k_\beta \frac{d}{2} +  c_\omega +  c_{b_\omega}$. Then, one has
	\begin{align*}
	\dot{\mathcal{L}}_p
	&\leq -\bar{e}_2\T\underbrace{\begin{bmatrix}
		\frac{d^2}{2}k_\beta -\bar{\mu}_2 \frac{d}{2}k_v & -\frac{1}{2} \bar{\mu}_2\bar{c}_5 \\
		-\frac{1}{2} \bar{\mu}_2\bar{c}_5  & \bar{\mu}_2
		\end{bmatrix}}_{H_3}\bar{e}_2  + \sqrt{2}\bar{c}_4 \bar{c}_p\|\tilde{b}_{\omega}\|  \|\bar{e}_2\|.
	\end{align*}
	To guarantee that the matrices $H_1,H_2$ and $H_3$ are positive definite, it is sufficient to pick $\bar{\mu}_2$ such that
	$
	0 < \bar{\mu}_2 < \min \left\{
	\frac{\sqrt{2d}}{\sqrt{k_v }}, \frac{k_\beta d^2}{dk_v+\frac{1}{2}\bar{c}_5^2}
	\right\}.
	$
	Hence, one has
	\begin{align}
	\dot{\mathcal{L}}_p
	& \leq - \lambda_{\min}^{H_3}\|\bar{e}_2\|^2 + \sqrt{2}\bar{c}_4 \bar{c}_p\|\bar{e}_1\|  \|\bar{e}_2\|  \nonumber \\
	& \leq  - \eta_2 \mathcal{L}_p + \eta_3\sqrt{\mathcal{L}_R}\sqrt{\mathcal{L}_p}
	\label{eqn:dot_L_p}, \quad  x'\in \mathcal{F}_c',
	\end{align}
	where $\eta_2:=  {\lambda_{\min}^{H_3}}/{\lambda_{\max}^{H_2} }>0$, $\eta_3:= { \sqrt{2}\bar{c}_4 \bar{c}_p}/{ \sqrt{\scalemath{0.8}{\lambda_{\min}^{H_1}}\bar{c}_1}}>0$, and we made use of the fact $\|\tilde{b}_{\omega}\| \leq \|\bar{e}_1\| \leq \frac{1}{\sqrt{\bar{c}_1}} \sqrt{\mathcal{L}_R} $.  Let $\zeta_1 := \sqrt{\mathcal{L}_R}$. $\zeta_2 :=\sqrt{\mathcal{L}_p}$ and  $\zeta := [\zeta_1, \zeta_2]\T$.
	From (\ref{eqn:dot_L_R}) and (\ref{eqn:dot_L_p}), one obtains
	\begin{align}
	\dot{\zeta} & \leq -H_4  \zeta, \quad H_4: = \begin{bmatrix}
	\bar{c}_3 & 0\\
	-\frac{\eta_3}{2} & \frac{\eta_2}{2}
	\end{bmatrix}.
	\end{align}
	One can easily verify that $H_4$ is positive definite. Let us consider the following Lyapunov function candidate:
	\begin{equation}
	\mathcal{L}'(x') = \mathcal{L}_R + \mathcal{L}_p = \|\zeta\|^2.
	\end{equation}
	Using the fact $|x'|_{\bar{\mathcal{A}}'}^2 = |\underline{\tilde{g}}|_I^2 + \|\tilde{b}_a\|^2 = \|\bar{e}_1\|^2 + \|\bar{e}_2\|^2$, from (\ref{eqn:c1_L_R_c2}), (\ref{eqn:dot_L_R}), (\ref{eqn:H1_L_P_H2}) and (\ref{eqn:dot_L_p}), one has
	\begin{align}
	&\underline{\alpha} |x'|_{\bar{\mathcal{A}}'}^2 \leq \mathcal{L}'(x') \leq \bar{\alpha} |x'|_{\bar{\mathcal{A}}'}^2 \label{eqn:inequalities2}\\
	&\dot{\mathcal{L}}'(x')  \leq -\zeta\T H_4 \zeta \leq - \lambda'_F \mathcal{L}'(x'), \quad x'\in \mathcal{F}_c',  \label{eqn:decrease_L_F2}
	\end{align}	
	where $\lambda'_F:= \lambda_{\min}^{H_4}$, $\underline{\alpha}: = \min\{\bar{c}_1, \lambda_{\min}^{H_1}\}$ and $\bar{\alpha}: = \max\{\bar{c}_2, \lambda_{\max}^{H_2}\}$. On the other hand, using the facts  $ \underline{\tilde{R}}^+=\underline{\tilde{R}}R_q,  \tilde{b}_\omega^+ = \tilde{b}_\omega, \tilde{b}_v^+ = \tilde{b}_v$ and
	$\underline{\tilde{p}}^+= \underline{\tilde{p}}$, one has
	\begin{align*}
	\mathcal{L}(x'^+) - \mathcal{L}(x')& = \|\zeta^+ \|^2  - \|\zeta \|^2  = \mathcal{L}_R^+ +  \mathcal{L}_p^+
	-  \mathcal{L}_R -  \mathcal{L}_p   \\
	& =  \tr(Q(I-\underline{\tilde{R}}^+))  - \tr(Q(I-\underline{\tilde{R}}))   - \bar{\mu}_1 \psi_a(\underline{\tilde{R}})\T \underline{\hat{R}} \tilde{b}_\omega + \bar{\mu}_1 \psi_a(\underline{\tilde{R}}^+)\T \underline{\hat{R}^+} \tilde{b}_\omega^+ \\
	& <  -\delta + \bar{\mu}_1 (\|\psi_a(\underline{\tilde{R}})\|     +   \|\psi_a(\underline{\tilde{R}}R_q)\|) \|\tilde{b}_\omega\| \\
	& \leq  -\delta + 4 \bar{\mu}_1   c_{b_\omega}.
	\end{align*}
	Choosing $\bar{\mu}_1 < \min \{\bar{\mu}_1^*,{\delta}/{4c_{b_\omega}}\}$ such that
	\begin{equation}
	\mathcal{L}(x'^+) \leq \mathcal{L}(x') -\bar{\delta}^*, \quad x'\in \mathcal{J}_c',  \label{eqn:decrease_L_J2}
	\end{equation}
	where $\bar{\delta}^* := -\delta + 4 \bar{\mu}_1    c_{b_\omega}>0$. In view of (\ref{eqn:decrease_L_F2}) and (\ref{eqn:decrease_L_J2}), one has $\mathcal{L}'(t,j) \leq  \mathcal{L}'(0,0) \exp(-\lambda'_F t)$ and $0\leq j\leq J_{\max}: = \left\lceil{{\mathcal{L}'}(0,0)}/{\bar{\delta}^*}~\right\rceil$. One can show that $\mathcal{L}'(x'^+)\leq \exp(-\lambda_J')\mathcal{L}'(x')$, with $\lambda_J': = -\ln (1-\bar{\delta}^*/\mathcal{L}'(0,0))$. Consequently, one obtains
	\begin{equation}
	\mathcal{L}'(t,j) \leq \exp(-2\lambda' (t+j)) \mathcal{L}'(0,0) \label{eqn:L_t_j2},
	\end{equation}
	where $\lambda': = \frac{1}{2}\min\{\lambda_F',\lambda_J'\}$. From (\ref{eqn:inequalities2}) and (\ref{eqn:L_t_j2}), for each $(t,j)\in \dom{x'}$ one has
	\begin{equation}
	|x'(t,j)|_{\bar{\mathcal{A}}'}\leq  k'\exp\left(-\lambda' (t+j) \right)|x'(0,0)|_{\bar{\mathcal{A}}'},
	\end{equation}
	where $k': = \sqrt{{\bar{\alpha}}/{\underline{\alpha}}}$.
	Using the same arguments as in the proof of Theorem \ref{theo:Theorem_1}, one can conclude that the solution to the hybrid system $\mathcal{H}'$ is complete. This completes the proof.

\end{document}